\documentclass[12pt]{amsart}
\usepackage{amscd,amssymb}
\usepackage{amsthm,amsmath,amssymb}
\usepackage[matrix,arrow,curve]{xy}
\usepackage{mathrsfs}
\usepackage{supertabular,longtable, rotating}
\usepackage{longtable}
\usepackage{multicol}
\usepackage{multirow}
\newcounter{cequation}[section]
\newtheorem{theorem}[cequation]{Theorem}

\newtheorem{proposition}[cequation]{Proposition}
\newtheorem{lemma}[cequation]{Lemma}
\newtheorem{corollary}[cequation]{Corollary}

\theoremstyle{definition}

\theoremstyle{remark}
\newtheorem{remark}[cequation]{Remark}
\makeatletter\@addtoreset{equation}{section}
\@addtoreset{cequation}{section}

\makeatother

\def\P {\mathbb{P}}
\def\C {\mathbb{C}}

\def\OO {\mathcal{O}}
\def\PP {\mathcal{P}}
\def\A {\mathfrak{A}}
\def\SS {\mathfrak{S}}

\def\PSp {\mathrm{PSp}_4(\mathbf{F}_3)}
\def\PSLseven {\mathrm{PSL}_2(\mathbf{F}_7)}
\def\PSLeleven {\mathrm{PSL}_2(\mathbf{F}_{11})}
\def\Pic {\mathrm{Pic}}
\def\Cl {\mathrm{Cl}}
\def\Sing {\mathrm{Sing}}
\def\Hom {\mathrm{Hom}}
\def\rkCl {\mathrm{rk\,Cl}}
\def\Sym {\mathrm{Sym}}

\def\ii {\mathrm{i}}
\def\vstrut {\vphantom{$\sqrt{\frac{A^A}{A_A}}$}}

\def\le {\leqslant}

\def\leq {\leqslant}

\title{Double quadrics with large automorphism groups}

\author{Victor Przyjalkowski, Constantin Shramov}

\thanks{This work was performed in Steklov Mathematical Institute and supported by the Russian Science Foundation under grant 14-50-00005.}

\address{
Steklov Mathematical Institute, 8 Gubkina st., Moscow, Russia, 119991
}
\email{victorprz@mi.ras.ru, costya.shramov@gmail.com}

\begin{document}

\begin{abstract}
We classify nodal Fano threefolds that are
double covers of smooth quadrics branched over intersections with quartics
acted on by finite simple non-abelian groups, and
study their rationality.
\end{abstract}

\maketitle

\section{Introduction}

In this paper we study double covers of three-dimensional
quadrics branched over intersections with quartics. For simplicity we
will sometimes call such varieties just \emph{double quadrics}.
They are Fano threefolds of Picard rank $1$, index $1$ and
anticanonical degree~$4$ (provided that their singularities
are sufficiently good). Double quadrics are degenerations
of quartic hypersurfaces (see e.g.~\cite[\S12.2]{IsPr99}),
and share many birational properties with the latter.
Rationality questions for double quadrics were studied
in \cite{Pukhlikov}, \cite[\S2.2]{IskovskikhPukhlikov},
\cite{Grinenko-1-pt}, \cite{Grinenko}, and~\cite{Shramov}.

Motivated by a study of finite subgroups of Cremona groups,
we are interested in the following problem. Given a finite group $G$
and a deformation family $\mathcal{X}$ of Fano varieties,
we would like to be able
to tell which of the varieties of $\mathcal{X}$ are acted
on by $G$, and which of them are $G$-Fano varieties
(see e.g.~\cite[\S1]{Pr15} or~\cite[\S1]{Pr16a}
for a definition).
The case that is currently most challenging is
when $\mathcal{X}$ is some family of Fano threefolds, and
$G$ is a simple non-abelian group (cf.~\cite{Prokhorov}).
In~\cite{CheltsovPrzyjalkowskiShramov}
this question was studied for quartic double solids with
an action of the icosahedral group~$\A_5$.
The purpose of this paper is
to study double quadrics from this point of view.

Our first result is a classification of possible
finite simple non-abelian groups acting on double quadrics
with mild singularities.

\begin{proposition}
\label{proposition:main}
Let $G$ be a finite simple non-abelian group.
Suppose that $G$ acts by automorphisms of a threefold
$X$ that is a double cover of an irreducible quadric
branched over an intersection with a quartic.
Then one has either $G\cong\A_6$, or $G\cong\A_5$.
In the former case the variety $X$ is unique, its singularities are
(isolated) ordinary double points, and $X$ is non-rational.
\end{proposition}

As one can see from Proposition~\ref{proposition:main},
the most interesting group we have to deal with is the
icosahedral group~$\A_5$. Our second result is a more refined
classification of double quadrics with icosahedral symmetry.

\begin{theorem}
\label{theorem:main}
There exist a two-parameter family $X_{\mu,\nu}$, $\mu,\nu\in\C$,
of threefolds,
and a threefold $X_{irr}$ such that $X_{\mu,\nu}$ and $X_{irr}$ are acted
on by the icosahedral group $\A_5$, and the following properties hold.
\begin{itemize}
\item[(i)]
Suppose that the group $\A_5$
acts by automorphisms of a threefold
$X$ that is a double cover of a smooth quadric
branched over an intersection with a quartic.
Suppose also that $X$ has at most isolated singularities.
Then either there is an $\A_5$-equivariant isomorphism
$X\cong X_{\mu,\nu}$ for some $\mu$ and $\nu$,
or there is an $\A_5$-equivariant isomorphism $X\cong X_{irr}$.

\item[(ii)] The varieties $X_{\mu,\nu}$ are non-rational up to a
finite number of possible exceptions.

\item[(iii)] A very general variety $X_{\mu,\nu}$
is not stably rational.

\item[(iv)] The variety $X_{irr}$ is non-rational.
\end{itemize}
\end{theorem}

\begin{remark}
We will see in~\S\ref{section:proj-irreducible} that the variety
$X_{irr}$ from Theorem~\ref{theorem:main} is actually
isomorphic to some variety $X_{\mu,\nu}$. However, there is no $\A_5$-equivariant
isomorphism between these two varieties, i.e. the corresponding two actions
of the group $\A_5$ are non-conjugate in the automorphism group
of~$X_{irr}$. To be more precise, for every $(\mu,\nu)$ the
space~\mbox{$H^0(X_{\mu,\nu}, -K_{X_{\mu,\nu}})$} is the (reducible) five-dimensional
permutation representation of the group~$\A_5$, while
$H^0(X_{irr}, -K_{X_{irr}})$ is the irreducible five-dimensional
representation of~$\A_5$.
Moreover, we will see that there is an action of the group~$\A_6$
on~$X_{irr}$, so that $X_{irr}$ also coincides with the variety described in
Proposition~\ref{proposition:main}. The two types of $\A_5$-actions on
$X_{irr}$ correspond to two non-conjugate embeddings of~$\A_5$
into~$\A_6$. As for the threefolds~$X_{\mu,\nu}$, we will see in Remark~\ref{remark:two-parameter}
that they indeed form a two-parameter family up to isomorphism.
\end{remark}

Theorem~\ref{theorem:main} gives a reasonable
(although still not complete) answer to
rationality questions for the family $X_{\mu,\nu}$ and for the variety
$X_{irr}$ that were asked in
\cite[Example~1.3.6]{CheltsovShramov} and
\cite[Example~1.3.7]{CheltsovShramov}, respectively.
There are still some exceptional cases among the varieties
$X_{\mu,\nu}$ one has to deal with
(see Corollaries~\ref{corollary:non-rationality-apart-of-30}
and~\ref{corollary:non-rationality-30}, and also Table~\ref{table:singularities}). It is possible that some
of these varieties are actually rational. However, we are not aware
of any rationality constructions for double quadrics with ordinary double
points that can be adapted for the case of large automorphism
groups. This is somehow opposite to the situation
with rational quartic threefolds that were extensively
studied in the literature, see e.g.~\cite{To33}, \cite{To35},
\cite{To36}, \cite{Pet98}, and \cite{CheltsovShramov2015}.
Maybe some of the above exceptional cases could be a starting point for a search of
rationality constructions for double quadrics.

The plan of the paper is as follows. In~\S\ref{section:preliminaries}
we make some preliminary remarks about double quadrics and their
automorphisms.
In~\S\ref{section:representations} we rule out most of the groups
in question via their representation theory, and prove (most of) Proposition~\ref{proposition:main}. In~\S\ref{section:proj-perm}
we classify singular complete intersections of a smooth quadric
and a quartic in the projectivization of the five-dimensional
permutation representation of the group~$\A_5$,
and make some conclusions about their non-rationality.
In particular, in~\S\ref{section:proj-perm}
we prove Theorem~\ref{theorem:main}(i) and give the most essential
part of the proof of Theorem~\ref{theorem:main}(ii).
Although the material
of~\S\ref{section:proj-perm} is totally computational, this is
the main technical part of the paper, modulo which all the rest
boils down to well-known constructions.
In~\S\ref{section:Artin-Mumford} we produce a construction
of double quadrics that are not stably rational similar
to a famous Artin--Mumford construction of
non-stably rational quartic double solids (see~\cite{ArtinMumford},
or rather~\cite[Appendix]{Aspinwall} for an approach that
we actually use), which proves Theorem~\ref{theorem:main}(iii). Although we do not know any literature
explicitly describing a construction like this for double
quadrics, it was definitely known for a while,
cf.~\cite[Remark~4.3]{IlievKatzarkovPrzyalkowski}, see also Remark~\ref{remark:vse-ukradeno-do-nas} below
for more references.
Finally, in~\S\ref{section:proj-irreducible} we discuss non-rationality
of the variety $X_{irr}$ using the approach of~\cite{Beauville-S6}
and use it to make final conclusions about non-rationality
of certain varieties~$X_{\mu,\nu}$. This proves Theorem~\ref{theorem:main}(iv)
and completes the proofs of Proposition~\ref{proposition:main} and Theorem~\ref{theorem:main}(ii).

\subsection*{Notation and conventions}
We work over the field $\C$ of complex numbers.
By~$\SS_n$ and~$\A_n$ we denote the symmetric
and the alternating group on~$n$ letters, respectively.
By~\mbox{$\Cl(X)$} we denote the group linear equivalence classes of Weil divisors
on a variety~$X$.
By an ordinary double point we always mean an isolated singularity
that is locally (analytically) isomorphic to a singularity
of a quadratic cone of an appropriate dimension.
By a very general element of an algebraic family we mean an element in a complement
to a countable union of Zariski closed subsets.

\subsection*{Acknowledgements}
We are grateful to A.\,Fonarev, S.\,Gorchinskiy, I.\,Netay, and Yu.\,Prokhorov for useful discussions.

\section{Preliminaries}
\label{section:preliminaries}

Let $\tau\colon X\to Q$ be double cover of an irreducible
three-dimensional
quadric $Q$ branched over a reduced surface $S$ that is cut out on
$Q$ by a quartic.
We will be interested in the case when
$X$ has terminal singularities.
In particular, in this situation
both $Q$ and $S$ have isolated singularities, so that $S$ is irreducible.
Note that the singularities of $X$ lie either over the vertex of
$Q$ (if $Q$ is a cone), or correspond to the singularities of $S$.

\begin{remark}
\label{remark:Pukhlikov}
Recall from \cite{Pukhlikov} (or~\cite[\S2.2]{IskovskikhPukhlikov})
that $X$ is non-rational
provided that $Q$ and $S$ are smooth.
Also, if $Q$ is a cone with an isolated singularity
and $S$ is a smooth surface that does not
pass through the vertex of $Q$ (so that $X$ has two singular points),
the threefold
$X$ is non-rational as well (see~\cite{Grinenko}).
Note however that if $Q$ is a cone with an isolated
singularity then $X$ cannot be a $G$-Fano variety
with respect to any simple non-abelian group $G$.
Therefore, we will be mostly interested in the case when
$Q$ is smooth and $S$ is singular.
\end{remark}

We will need the following general auxiliary result.

\begin{lemma}\label{lemma:quadric-vs-P4}
Let $Y\subset\P^n$ be a linearly normal Gorenstein variety
that is not contained in a hyperplane.
Suppose that a group $G$ acts on $Y$ so that the class of the line bundle
$$\OO_Y(1)=\OO_{\P^n}(1)\vert_Y$$
in $\Pic(Y)$ is $G$-invariant.
Suppose that $\omega_Y\cong \OO_Y(i)$ for some (non-zero) integer
$i$, where $\omega_Y$ is the canonical line bundle on $Y$,
and suppose that the numbers $i$ and $n+1$ are coprime.
Then the line bundle $\OO_Y(1)$ has a $G$-equivariant structure.
\end{lemma}
\begin{proof}
Since the class of $\OO_Y(1)$ in $\Pic(Y)$ is $G$-invariant,
there is an action of the group~$G$ on
$$\P^n\cong \P\big(H^0(Y,\OO_Y(1))^\vee\big)$$
that agrees with the initial action of $G$ on $Y$.
The line bundle $\omega_{\P^n}\cong\OO_{\P^n}(-n-1)$ has a $G$-equivariant structure,
and the same holds for its restriction to $Y$.
Also, the line bundle~$\omega_Y$ has a $G$-equivariant structure.
By assumption there are integers $a$ and $b$ such that
$$ai-b(n+1)=1.$$
This gives
$$
\OO_Y(1)\cong\omega_Y^{\otimes a}\otimes \omega_{\P^n}^{\otimes b}\vert_Y.
$$
Therefore, $\OO_Y(1)$ has a $G$-equivariant structure as well.
\end{proof}

\begin{lemma}[{cf. \cite[Remark~2.2]{CheltsovPrzyjalkowskiShramov}}]
\label{lemma:lifting-shifting}
Suppose that $X$ admits a faithful action of a finite group $G$.
Then there is a natural action of $G$ on $\P^4$ such that $Q$ is $G$-invariant, and the surface $S$ is
cut out on $Q$ by a $G$-invariant quartic in $\P^4$.
Vice versa, fix an action of $G$ on $Q$ such that the surface $S$ is
invariant, and suppose that $S$ corresponds to a \textup{trivial}
subrepresentation of $G$ in $\Sym^4\big(H^0(Q,\OO_Q(H))\big)$,
where $H$ is a hyperplane section of $Q$.
Then $G$ acts on the threefold $X$.
\end{lemma}
\begin{proof}
To prove the first assertion note that the morphism
$\tau$ is given by the linear system~\mbox{$|-K_X|$}.
Thus the group $G$ acts on $Q$ so that $S$ is $G$-invariant.
One has $Q\subset\P^4$, and there is a $G$-equivariant
identification of the linear systems
$|\OO_{\P^4}(1)|$ and $|\OO_{\P^4}(1)\vert_Q|$. In particular, $S$ is cut out on $Q$
by a $G$-invariant quartic in $\P^4$.

To prove the second assertion suppose that $G$ acts on $Q$ so that
$S$ is $G$-invariant. One has $Q\subset\P^4$, and $\P^4$ is identified with
a projectivization of the $G$-representation
$$U^\vee=H^0\big(Q,\OO_Q(H)\big)^\vee$$
by Lemma~\ref{lemma:quadric-vs-P4}.
The threefold $X$ has a natural embedding to the projectivization of
the vector bundle
$$\mathcal{E}=\mathcal{O}_{Q}\oplus\mathcal{O}_{Q}(2H).$$
The vector bundle $\mathcal{E}$ has a $G$-equivariant structure by Lemma~\ref{lemma:quadric-vs-P4}.
By assumption the group $G$ acts trivially on the
one-dimensional subspace of $\Sym^4(U)$
corresponding to $S$, so that $X$ is given by an equation in the projectivization of~$\mathcal{E}$.
\end{proof}

\section{Representations}
\label{section:representations}

\begin{lemma}\label{lemma:Prokhorov}
Let $G$ be a finite simple non-abelian group.
Suppose that $G$ acts by automorphisms of a threefold
$X$ that is a double cover of an irreducible quadric
branched over an intersection
with a quartic.
Then either $G\cong\A_6$, or $G\cong\A_5$. In the former case
the variety $X$ is unique.
\end{lemma}
\begin{proof}
By Lemmas~\ref{lemma:lifting-shifting} and~\ref{lemma:quadric-vs-P4}
there exists a faithful five-dimensional representation $V$ of the group $G$,
and there is an irreducible (reduced) $G$-invariant quadric in $\P(V)$.
Suppose that $G$ is not isomorphic to $\A_5$.
We know from~\cite[\S8.5]{Feit}
that $G$ is one of the groups $\PSp$, $\PSLeleven$, $\PSLseven$, or $\A_6$.

Assume that $G\cong\PSp$. Then $V$ is one of the two
irreducible five-dimensional representations of $G$
(see~\cite[p.~27]{Atlas}). However, there are no
$G$-invariant quadrics in~\mbox{$\P(V)$}, see e.g.~\cite{Bu91}.

Assume that $G\cong\PSLeleven$.
Then $V$ is one of the two
irreducible five-dimensional representations of $G$, see~\cite[p.~7]{Atlas}. As in the previous case, there are no
$G$-invariant quadrics in~\mbox{$\P(V)$}.

Assume that $G\cong\PSLseven$.
Then $V\cong I\oplus I\oplus V_3$, where $I$ is a trivial
representation and $V_3$ is one of the two irreducible
three-dimensional representations of $G$, see~\cite[p.~3]{Atlas}.
There are no one-dimensional $G$-subrepresentations in $\Sym^2(V_3^\vee)$,
see e.g.~\cite{Klein}. Thus, every $G$-invariant quadric in $\P(V)$ is reducible
(or non-reduced).

Finally, assume that $G\cong\A_6$. Then
$G$ has two faithful five-dimensional
representations, and both of them are irreducible
(see e.g.~\cite[p.~5]{Atlas}). Moreover, the images of
$G$ in $\mathrm{PGL}_5(\C)$ under these two representations
are conjugate.
One can check that $\Sym^2(V^\vee)$ contains a unique
trivial subrepresentation,
and
$$
\dim\mathrm{Hom}\big(I,\Sym^4(V^\vee)\big)=2,
$$
where $I$ is the trivial representation of $G$ (which is also its
only one-dimensional representation).
One of the above trivial subrepresentations corresponds to a square of the
$\A_6$-invariant quadratic form on $V$ (that is unique up to scaling).
Therefore,
$X$ is uniquely defined in this case.
\end{proof}

\begin{remark}
Let $X$ be the (unique) threefold
that is a double cover of an irreducible quadric
branched over an intersection
with a quartic such that $X$ has a non-trivial action of the
group~$\A_6$ (see Lemma~\ref{lemma:Prokhorov}).
We will see later in Section~\ref{section:proj-irreducible}
that the singularities of $X$ are ordinary double points.
\end{remark}

Keeping in mind Lemma~\ref{lemma:Prokhorov}, in the rest of the paper
we will work with double quadrics that admit a faithful action of the
icosahedral group~$\A_5$.
Denote by $I$ the trivial representation of the group $\A_5$.
Let $W_3$ and $W_3^\prime$
be the two irreducible three-dimensional representations of $\A_5$,
and let $W_4$ and $W_5$ be the irreducible four-dimensional and five-dimensional representations of~$\A_5$,
respectively (see e.g.~\cite[p.~2]{Atlas}). Note that $I$, $W_4$, and $W_5$ can be also
considered as representations of the group $\SS_5$.

There are four faithful five-dimensional representations of the
icosahedral group $\A_5$, namely:
\mbox{$I\oplus I\oplus W_3$}, $I\oplus I\oplus W_3^\prime$, $I\oplus W_4$, and $W_5$.

\begin{remark}
\label{remark:sym-sym}
Put $U=I\oplus I\oplus W_3$ or $U=I\oplus I\oplus W_3^\prime$.
Then
$$\dim\Hom\big(I,\Sym^2(U^\vee)\big)=4, \quad
\dim\Hom\Big(I,\Sym^4(U^\vee)\big)=9.$$
Let $x_0$ and $x_1$ be coordinates in the $\A_5$-subrepresentation
$I\oplus I$, and let $y_0,y_1,y_2$ be coordinates in the $\A_5$-subrepresentation
$W_3$ or $W_3^\vee$.
It is well-known that up to scaling there is a unique
form of degree $2$ in $y_0,y_1,y_2$
that is preserved by $\A_5$.
We may assume that $y_0,y_1,y_2$ are chosen so that
this form is~\mbox{$y_0^2+y_1^2+y_2^2$}.
One can consider~\mbox{$x_0,x_1,y_0,y_1,y_2$}
as homogeneous coordinates on $\P^4=\P(U)$.
Every $\A_5$-invariant quadric in $\P^4$ is given by equation
$$
F_2(x_0,x_1)+\alpha (y_0^2+y_1^2+y_2^2)=0,
$$
where $F_2$ is a form of degree $2$ and $\alpha\in\C$.
In particular, this quadric contains the curve $C$ given by
equations
$$
x_0=x_1=y_0^2+y_1^2+y_2^2=0.
$$
Every $\A_5$-invariant quartic in $\P^4$ is given by equation
$$
G_4(x_0,x_1)+ G_2(x_0,x_1)(y_0^2+y_1^2+y_2^2)
+\beta (y_0^2+y_1^2+y_2^2)^2=0,
$$
where $G_i$ is a form of degree $i$ and $\beta\in\C$.
In particular, it is singular along the curve $C$.
Therefore, every complete intersection of a quadric and
a quartic in $\P(U)$ has non-isolated singularities.
\end{remark}

By Remark~\ref{remark:sym-sym}
to classify three-dimensional double quadrics with isolated singularities that admit an
icosahedral symmetry it remains to
consider the $\A_5$-representations
$I\oplus W_4$ and $W_5$. We will do this in the next sections.

\begin{remark}
In general, it is an interesting problem to classify Fano varieties with an action of
some relatively large finite group, for example the icosahedral group~$\A_5$.
Apart from applications to classification of finite subgroups of Cremona groups,
this often leads to really beautiful geometric constructions.
We do not expect many examples like this among Fano threefolds of large anticanonical
degree. However, it is possible that some interesting cases can arise among
$\mathbb{Q}$-Fano threefolds of large index (cf.~\cite{Pr16b})
or smooth Fano fourfolds of large anticanonical degree (cf.~\cite{Kuz15}).
\end{remark}

\section{Projectivization of the permutation representation}
\label{section:proj-perm}

Consider the vector space $I\oplus W_4$ as a permutation representation
of the groups $\A_5$ and~$\SS_5$, and put $\P^4=\P(I\oplus W_4)$.
Let $x_0,\ldots,x_4$
be homogeneous coordinates in $\P^4$ that
are permuted by~$\SS_5$.
Put
\begin{equation*}
\sigma_k(x_0,\ldots,x_4)=x_0^k+\ldots+x_4^k.
\end{equation*}
It is easy to see that every reduced $\A_5$-invariant quadric in
$\P^4$ is given by equation
\begin{equation}\label{eq:quadric-lambda}
\sigma_2(x_0,\ldots,x_4)+\lambda\sigma_1(x_0,\ldots,x_4)^2=0
\end{equation}
for some $\lambda\in\C$.

\begin{lemma}\label{lemma:quadric-lambda}
A quadric $Q$ given by equation~\eqref{eq:quadric-lambda}
is singular if and only if $\lambda=-\frac{1}{5}$. If
it is non-singular, there is an $\SS_5$-equivariant
linear change of coordinates such that after it $Q$
is given by equation~\eqref{eq:quadric-lambda} with $\lambda=0$.
\end{lemma}
\begin{proof}
The quadratic form associated to the quadric~\eqref{eq:quadric-lambda} is given by the matrix
$$
M=\left(
  \begin{array}{ccccc}
    \lambda+1 & \lambda & \lambda & \lambda & \lambda \\
    \lambda & \lambda+1 & \lambda & \lambda & \lambda \\
    \lambda & \lambda & \lambda+1 & \lambda & \lambda \\
    \lambda & \lambda & \lambda & \lambda+1 & \lambda \\
    \lambda & \lambda & \lambda & \lambda & \lambda+1 \\
  \end{array}
\right).
$$
One has $\det M=5\lambda+1$, so the matrix is degenerate
only for $\lambda=-\frac{1}{5}$, and this is the only case when the
quadric is singular.
Let $\lambda\neq -\frac{1}{5}$ and let $\alpha$ be a root of the equation
\begin{equation}\label{eq:alpha}
5\alpha^2+2\alpha=\lambda.
\end{equation}
Consider a linear change of coordinates
putting
$$x_i'=x_i+\alpha\sigma_1(x_0,\ldots,x_4).$$
Note that it is indeed an invertible change of coordinates,
i.e. the corresponding matrix is non-degenerate.
One has
\begin{multline*}
\sigma_2(x_0',\ldots,x_4')
=\sigma_2(x_0,\ldots,x_4)+2\alpha\sigma_1^2(x_0,\ldots,x_4)+
5\alpha^2\sigma_1^2(x_0,\ldots,x_4)=\\
=\sigma_2(x_0,\ldots,x_4)+\lambda\sigma_1^2(x_0,\ldots,x_4).
\end{multline*}
\end{proof}

Keeping in mind Lemma~\ref{lemma:quadric-lambda} and
Remark~\ref{remark:Pukhlikov}, in the rest of this section
we will ignore the case $\lambda=-\frac{1}{5}$ and
will denote by $Q$ the quadric given by equation
\begin{equation}\label{eq:quadric}
\sigma_2(x_0,\ldots,x_4)=0,
\end{equation}
i.e. by equation~\eqref{eq:quadric-lambda} with $\lambda=0$.
Every irreducible reduced $\A_5$-invariant intersection of $Q$ with a quartic in $\P^4$ is
given by equations~\eqref{eq:quadric} and
\begin{equation}\label{eq:quartic}
\sigma_4(x_0,\ldots,x_4)+4\mu \sigma_3(x_0,\ldots,x_4)\sigma_1(x_0,\ldots,x_4)
+\nu \sigma_1(x_0,\ldots,x_4)^4=0
\end{equation}
for some $\mu, \nu\in\C$.
We will denote such intersection by $S_{\mu,\nu}$.

\begin{remark}
Double covers of $Q$ branched over $S_{\mu,\nu}$ form a two-parameter family of double
quadrics. Thus, Lemma~\ref{lemma:Prokhorov} and Remark~\ref{remark:sym-sym} imply Theorem~\ref{theorem:main}(i).
\end{remark}

Writing down partial derivatives, we see that a point $P\in\P^4$ is a singular point
of the surface $S_{\mu,\nu}$ if and only if the form $\sigma_2$ and
the determinants of the matrices
\begin{equation}\label{equation:matrix-2x2}
M_{ij}=\left(
\begin{array}{cc}
x_i^3+\mu \sigma_3+3\mu x_i^2\sigma_1+
\nu\sigma_1^3 &
x_j^3+\mu \sigma_3+3\mu x_j^2\sigma_1+
\nu\sigma_1^3\\
x_i & x_j
\end{array}
\right)
\end{equation}
for all $0\le i<j\le 4$ vanish at $P$.

\begin{lemma}
\label{lemma:fundamental}
Let $P=(1:x_1:x_2:x_3:x_4)$ be a singular point of a surface $S_{\mu,\nu}$.
Then one of the following cases occurs.
\begin{itemize}
  \item[(i)] Up to permutation of coordinates one has $P=(1:1:\ii:\ii:0)$.
  \item[(ii)] Up to permutation of coordinates one has $P=(1:0:0:0:\ii)$.
  \item[(iii)] Up to permutation of coordinates one has $P=(1:1:1:a:b)$,
for some $a,b\in\C$ such that $a^2+b^2+3=0$.
  \item[(iv)] Up to permutation of coordinates one has $P=(1:a:a:b:b)$,
for some $a,b\in\C$ such that $2a^2+2b^2+1=0$.
\end{itemize}
\end{lemma}

\begin{proof}
Rewrite~\eqref{equation:matrix-2x2} as
\begin{equation}
\label{equation:M_ij}
M_{0,i}=(x_i-1)\cdot \left(\left(\sigma_3(P)-3x_i\sigma_1(P) \right)\mu+
\sigma_1(P)^3\nu-x_i(x_i+1)\right).
\end{equation}
We are going to show that
$|\{1,x_1,x_2,x_3,x_4\}|\leq 3$, i.e. the number of non-equal
coordinates that are also not equal to $1$ among $x_1,x_2,x_3,x_4$ is
at most~$2$.
Indeed, suppose that $x_1$, $x_2$, and~$x_3$ are three non-equal coordinates
that are also not equal to $1$.
Consider the following cases.

\begin{itemize}
\item One has $\sigma_1(P)=\sigma_3(P)=0$. From the vanishing of
the determinant of $M_{0,i}$ one gets
either $x_i=1$ or $x_i(x_i+1)=0$
for every $i=1,\ldots,4$. This means that $x_i\in\{0,1,-1\}$,
so that condition~\eqref{eq:quadric} fails.

\item One has $\sigma_1(P)=0$, $\sigma_3(P)\neq 0$.
From the vanishing of the determinant of $M_{0,i}$ one gets
\begin{equation*}
\label{equation:mu-Sigma-0}
\mu=\frac{x_i(x_i+1)}{\sigma_3(P)}
\end{equation*}
for every $i=1,2,3$.
This implies
$$0=x_i(x_i+1)-x_j(x_j+1)=(x_i-x_j)(x_i+x_j+1)$$
for $i,j=1,2,3$, so that
$x_i+x_j+1=0$,
which gives a contradiction with the assumption that all $x_k$, $k=1,2,3$, are different.
\item One has $\sigma_1(P)\neq 0$.
From the equalities $\det M_{0,i}=\det M_{0,j}$ for $i,j=1,2,3$ one gets
\begin{equation}
\label{equation:mu-Sigma-ne0}
\mu=-\frac{x_i+x_j+1}{3\sigma_1(P)}.
\end{equation}
Hence
$$x_1+x_2+1=x_1+x_3+1=x_2+x_3+1,$$
which again gives a contradiction with the condition that all $x_k$, $k=1,2,3$, are different.
\end{itemize}

Thus one can assume that $x_1=a$, $x_2=b$,
and each of the coordinates~$x_3$ and~$x_4$ equals either~$1$, or~$a$, or~$b$.
This together with equation~\mbox{$\sigma_2(P)=0$}
gives possibilities (i)--(iv) for~$P$.
\end{proof}

\begin{corollary}\label{corollary:sing-on-sigma1}
Let $P=(x_0,\ldots,x_4)$ be a point such that $\sigma_1(P)=0$.
Then $P$ is not a singular point of any surface $S_{\mu,\nu}$.
\end{corollary}
\begin{proof}
Suppose that $P$ is a singular point of
a surface $S_{\mu,\nu}$.
Recalling equations~\eqref{eq:quadric} and~\eqref{eq:quartic},
we see that
\begin{equation}\label{eq:sigma-123}
\sigma_1(P)=\sigma_2(P)=\sigma_4(P)=0.
\end{equation}
On the other hand, the point $P$ must be
of one of the types listed in Lemma~\ref{lemma:fundamental}.
For the cases~(i) and~(ii) the assumption $\sigma_1(P)=0$ fails.
In the remaining two cases it is straightforward to check that the system of
equations~\eqref{eq:sigma-123} has no solutions.
\end{proof}

\begin{corollary}\label{corollary:isolated-singularities}
Any surface $S_{\mu,\nu}$ has isolated singularities.
\end{corollary}
\begin{proof}
If $S_{\mu,\nu}$ has non-isolated singularities, then there is a
singular point $P$ of $S_{\mu,\nu}$ such that $\sigma_1(P)=0$.
The latter is impossible by Corollary~\ref{corollary:sing-on-sigma1}.
\end{proof}

We introduce the following notation:
\begin{itemize}
\item $\Sigma_5^+$ denotes the $\SS_5$-orbit of the point $P=(1:1:1:1:2\ii)$;

\item $\Sigma_5^-$ denotes the $\SS_5$-orbit of the point $P=(1:1:1:1:-2\ii)$;

\item $\Sigma_{10}^+$ denotes the $\SS_5$-orbit of the point
$P=(1:1:1:\sqrt{\frac{3}{2}}\ii:\sqrt{\frac{3}{2}}\ii)$;

\item $\Sigma_{10}^-$ denotes the $\SS_5$-orbit of the point
$P=(1:1:1:-\sqrt{\frac{3}{2}}\ii:-\sqrt{\frac{3}{2}}\ii)$;

\item $\Sigma_{20}$ denotes the $\SS_5$-orbit of the point $P=(1:0:0:0:\ii)$;

\item $\Sigma_{20}^{a,b}$ denotes the $\SS_5$-orbit of the point $P=(1:1:1:a:b)$, where $a, b\in\C$ are such that
$a\neq b$, $a\neq 1$, and $b\neq 1$; note that the pairs $(a,b)$ and
$(b,a)$ give the same $\SS_5$-orbit, and thus only a non-ordered pair
is important here;

\item $\Sigma_{30}$ denotes the $\SS_5$-orbit of the point $P=(1:1:\ii:\ii:0)$;

\item $\Sigma_{30}^{a,b}$ denotes the $\SS_5$-orbit of the point
$P=(1:a:a:b:b)$, where $a, b\in\C$ are such that
$a\neq b$, $a\neq 1$, and $b\neq 1$;
as before, the pairs $(a,b)$ and
$(b,a)$ give the same $\SS_5$-orbit, and thus only a non-ordered pair
is important here.
\end{itemize}

\begin{remark}
One has
$$|\Sigma_k^+|=|\Sigma_k^-|=k, \quad
|\Sigma_k|=k, \quad
|\Sigma_k^{a,b}|=k.$$
Moreover, one can check that apart from the unique $\SS_5$-fixed
point in $\P^4$ that corresponds to the trivial
subrepresentation (and that is not contained in the quadric~$Q$),
the orbits
$\Sigma_5^+$, $\Sigma_5^-$, $\Sigma_{10}^+$, $\Sigma_{10}^-$,
$\Sigma_{20}$, $\Sigma_{20}^{a,b}$, $\Sigma_{30}$, and
$\Sigma_{30}^{a,b}$ are the only $\SS_5$-orbits
of length less than~$60$ in~$\P^4$.
Note also that $\Sigma_5^+$, $\Sigma_5^-$,
$\Sigma_{10}^+$, $\Sigma_{10}^-$, and
$\Sigma_{20}$ are contained in the closure of the union of
the $\SS_5$-orbits~$\Sigma_{20}^{a,b}$. Similarly,
$\Sigma_5^+$, $\Sigma_5^-$,
$\Sigma_{10}^+$, $\Sigma_{10}^-$, and
$\Sigma_{30}$ are contained in the closure of the union
of the $\SS_5$-orbits~$\Sigma_{30}^{a,b}$.
\end{remark}

\begin{corollary}
\label{corollary:nu}
Let $P=(x_0,\ldots,x_4)\in S_{\mu,\nu}$ be a point such that $\sigma_1(P)\neq 0$.
Then $P$ is singular on $S_{\mu,\nu}$ if and only if one of the following cases occurs.
\begin{itemize}
  \item[(i)] Up to permutation of coordinates one has $P=(1:0:0:0:\ii)$, so that $P\in \Sigma_{20}$;
in this case
$$
\mu= -\frac{1}{3},\ \nu= -\frac{1}{6}.
$$

  \item[(ii)] Up to permutation of coordinates one has $P=(1:1:\ii:\ii:0)$, so that $P\in \Sigma_{30}$;
in this case
$$
\mu= -\frac{1}{6},\ \nu= -\frac{1}{48}.
$$

\item[(iii)] Up to permutation of coordinates
one has $P=(1:1:1:1:2\ii)$, so that $P\in\Sigma_5^+$.
Then~$P$ is a singular point of a surface $S_{\mu,\nu}$
if and only if
\begin{equation}
\label{eq:sigma5plus_mu_nu}
\nu=
\left(\frac{8}{25}+\frac{6}{25}\ii\right)\mu+\left(\frac{7}{500}+\frac{6}{125}\ii\right).
\end{equation}

\item[(iii')] Up to permutation of coordinates
one has $P=(1:1:1:1:-2\ii)$, so that $P\in\Sigma_5^-$.
Then~$P$ is a singular point of a surface $S_{\mu,\nu}$
if and only if
\begin{equation}
\label{eq:sigma5minus_mu_nu}
\nu=
\left(\frac{8}{25}-\frac{6}{25}\ii\right)\mu+\left(\frac{7}{500}-\frac{6}{125}\ii\right).
\end{equation}

\item[(iv)] Up to permutation of coordinates
one has
$P=\left(1:1:1:\sqrt{\frac{3}{2}}\ii:\sqrt{\frac{3}{2}}\ii\right)$,
so that $P\in\Sigma_{10}^+$.
Then~$P$ is a singular point of a surface $S_{\mu,\nu}$
if and only if
\begin{equation}
\label{eq:sigma10plus_mu_nu}
\nu=\left(\frac{8}{25}+\frac{2\sqrt{6}}{75}\ii\right)\mu+\frac{23}{750}+\frac{2\sqrt{6}}{375}\ii.
\end{equation}

\item[(iv')]
Up to permutation of coordinates one has
$P=\left(1:1:1:-\sqrt{\frac{3}{2}}\ii:-\sqrt{\frac{3}{2}}\ii\right)$,
so that $P\in\Sigma_{10}^-$.
Then~$P$ is a singular point of a surface $S_{\mu,\nu}$
if and only if
\begin{equation}
\label{eq:sigma10minus_mu_nu}
\nu=\left(\frac{8}{25}-\frac{2\sqrt{6}}{75}\ii\right)\mu+\frac{23}{750}-\frac{2\sqrt{6}}{375}\ii.
\end{equation}

  \item[(v)] Up to permutation of coordinates one has $P=(1:1:1:a:b)$
or~\mbox{$P=(1:a:a:b:b)$} with $a\neq 1$, $b\neq 1$, $a\neq b$; in this case
\begin{equation}\label{eq:mu-nu}
\mu = -\frac{a+b+1}{3\sigma_1(P)},\
\nu=\frac{\sigma_3(P)(a+b+1)-3ab\sigma_1(P)}{3\sigma_1(P)^4}.
\end{equation}
\end{itemize}
\end{corollary}

\begin{proof}
The value of $\mu$ is given by~\eqref{equation:mu-Sigma-ne0},
and the value of $\nu$ can be found from~\eqref{equation:M_ij}.
\end{proof}

Using Corollary~\ref{corollary:nu}, we derive the following facts.

\begin{corollary}\label{corollary:1-1-1-a-b}
Suppose that $P=(1:1:1:a:b)$ and $\sigma_2(P)=a^2+b^2+3=0$.
Suppose also that $a\neq 1$, $b\neq 1$, and $a\neq b$,
i.e. one has $P\in\Sigma_{20}^{a,b}$.
Then one of the following cases occurs.
\begin{enumerate}
\item[(i)] Up to permutation of coordinates one has
$$P=\left(1:1:1:-\frac{3}{2}+\frac{\sqrt{15}}{2}\ii:
-\frac{3}{2}-\frac{\sqrt{15}}{2}\ii\right).$$
Then $\sigma_1(P)=0$, and $P$ is not a singular
point of any surface~$S_{\mu,\nu}$
(cf. Corollary~\ref{corollary:sing-on-sigma1}).

\item[(ii)]
One has $\sigma_1(P)\neq 0$.
Then~$P$ is a singular point of a surface $S_{\mu,\nu}$
if and only if
\begin{equation*}
\mu = -\frac{a+b+1}{3(a+b+3)},\
\nu=\frac{(a^3+b^3+3)(a+b+1)-3ab(a+b+3)}{3(a+b+3)^4}.
\end{equation*}
\end{enumerate}
\end{corollary}

\begin{remark}\label{remark:quartic-curve-1}
In case~(ii) of Corollary~\ref{corollary:1-1-1-a-b} one has
$$
\nu=-\frac{405}{4}\mu^4-81\mu^3-27\mu^2-\frac{1}{4}.
$$
\end{remark}

\begin{corollary}\label{corollary:a-a-b-b-1}
Suppose that $P=(1:a:a:b:b)$, and $\sigma_2(P)=2a^2+2b^2+1=0$.
Suppose also that $a\neq 1$, $b\neq 1$, and $a\neq b$,
i.e. one has $P\in\Sigma_{30}^{a,b}$.
Then one of the following cases occurs.
\begin{enumerate}
\item[(i)] Up to permutation of coordinates one has
$$P=\left(-1:\frac{1}{4}+\frac{\sqrt{5}}{4}\ii:\frac{1}{4}+\frac{\sqrt{5}}{4}\ii:\frac{1}{4}-\frac{\sqrt{5}}{4}\ii:\frac{1}{4}-\frac{\sqrt{5}}{4}\ii
\right).$$
Then $\sigma_1(P)=0$, and $P$ is not a singular
point of any surface~$S_{\mu,\nu}$
(cf. Corollary~\ref{corollary:sing-on-sigma1}).

\item[(ii)]
One has $\sigma_1(P)\neq 0$.
Then~$P$ is a singular point of a surface $S_{\mu,\nu}$
if and only if
\begin{equation*}
\mu = -\frac{a+b+1}{3(2a+2b+1)},\
\nu=\frac{(2a^3+2b^3+1)(a+b+1)-3ab(2a+2b+1)}{3(2a+2b+1)^4}.
\end{equation*}
\end{enumerate}
\end{corollary}

\begin{remark}\label{remark:quartic-curve-2}
In case~(ii) of Corollary~\ref{corollary:a-a-b-b-1} one has
$$
\nu=405\mu^4+324\mu^3+99\mu^2+14\mu+\frac{3}{4}.
$$
\end{remark}

Now we are ready to give a complete description
of singular loci of the surfaces~$S_{\mu,\nu}$ and the corresponding
double covers. To do this we will need some additional notation.

\begin{itemize}
\item
Let $\mu_{5,1}^+$ and $\mu_{5,2}^+$ be two different roots of the equation
$$
894825\mu^2+(126510+149670\ii)\mu-1249+9382\ii=0
$$
and define $\nu_{5,1}^+$ and $\nu_{5,2}^+$ by formula~\eqref{eq:sigma5plus_mu_nu}.

\item
Let
$\mu_{5}^{ab,+}=-\frac{1}{5}-\frac{1}{15}\ii,\ \nu_{5}^{ab,+}=-\frac{17}{500}-\frac{8}{375}\ii$.

\item
Let $\mu_{5,1}^-$ and $\mu_{5,2}^-$ be two different roots of the equation
$$
894825\mu^2+(126510-149670\ii)\mu-1249-9382\ii=0
$$
and define $\nu_{5,1}^-$ and $\nu_{5,2}^-$ by formula~\eqref{eq:sigma5minus_mu_nu}.

\item
Let
$\mu_{5}^{ab,-}=-\frac{1}{5}+\frac{1}{15}\ii,\ \nu_{5}^{ab,-}=-\frac{17}{500}+\frac{8}{375}\ii$.

\item
Let $\mu_{10,1}^+$ and $\mu_{10,2}^+$ be two different roots of the equation
$$
216090\mu^2+(34140+1485\sqrt{6}\ii)\mu+1556+93\sqrt{6}\ii=0
$$
and define $\nu_{10,1}^+$ and $\nu_{10,2}^+$ by formula~\eqref{eq:sigma10plus_mu_nu}.

\item
Let $\mu_{20\to 10}^+=-\frac{1}{5}-\frac{2\sqrt{6}}{45}\ii,\ \nu_{20\to 10}^+=-\frac{59}{2250}-\frac{16\sqrt{6}}{1125}\ii$.

\item
Let $\mu_{30\to 10}^+=-\frac{1}{5}+\frac{\sqrt{6}}{90}\ii,\ \nu_{30\to 10}^+=-\frac{79}{2250}+\frac{4\sqrt{6}}{1125}\ii$.

\item
Let $\mu_{10,1}^-$ and $\mu_{10,2}^-$ be two different roots of the equation
$$
216090\mu^2+(34140-1485\sqrt{6}\ii)\mu+1556-93\sqrt{6}\ii=0
$$
and define $\nu_{10,1}^-$ and $\nu_{10,2}^-$ by formula~\eqref{eq:sigma10minus_mu_nu}.

\item
Let $\mu_{20\to 10}^-=-\frac{1}{5}+\frac{2\sqrt{6}}{45}\ii,\ \nu_{20\to 10}^-=-\frac{59}{2250}+\frac{16\sqrt{6}}{1125}\ii$.

\item
Let $\mu_{30\to 10}^-=-\frac{1}{5}-\frac{\sqrt{6}}{90}\ii,\ \nu_{30\to 10}^-=-\frac{79}{2250}-\frac{4\sqrt{6}}{1125}\ii$.

\item Let $a_{20,i}$ be the roots of the equation
\begin{equation}\label{eq:a-20-i}
4a_{20,i}^8+16a_{20,i}^7+56a_{20,i}^6+116a_{20,i}^5+217a_{20,i}^4+266a_{20,i}^3+257a_{20,i}^2+172a_{20,i}+52=0\end{equation}
and put
\begin{multline*}
b_{20,i}= \frac{98}{215}a_{20,i}^7+\frac{314}{215}a_{20,i}^6+\frac{1094}{215}a_{20,i}^5
+\frac{376}{43}a_{20,i}^4+\frac{1401}{86}a_{20,i}^3+\frac{6567}{430}a_{20,i}^2+\frac{2874}{215}a_{20,i}+\frac{1271}{215}.
\end{multline*}
Note that $b_{20,i}$ is also a root of equation~\eqref{eq:a-20-i}, and there are only four
possible non-ordered pairs $(a_{20,i},b_{20,i})$; thus we will assume that
$i=1,\ldots,4$.
Define $\mu_{20,i}$ and $\nu_{20,i}$ by formula~\eqref{eq:mu-nu}.

\item Let $a_{30,i}$ be the roots of equation
\begin{multline}\label{eq:a-30-i}
512a_{30,i}^8+2048a_{30,i}^7+4608a_{30,i}^6+5952a_{30,i}^5+5698a_{30,i}^4+\\
+3740a_{30,i}^3+1765a_{30,i}^2+602a_{30,i}+163=0
\end{multline}
and put
\begin{multline*}
b_{30,i}=-\frac{22544}{12639}a_{30,i}^7-\frac{26992}{4213}a_{30,i}^6-\frac{174304}{12639}a_{30,i}^5-\frac{204442}{12639}a_{30,i}^4
-\frac{1036395}{67408}a_{30,i}^3-\\
-\frac{1670095}{202224}a_{30,i}^2-\frac{1606787}{404448}a_{30,i}-\frac{480445}{404448}.
\end{multline*}
Note that $b_{30,i}$ is also a root of equation~\eqref{eq:a-30-i}, and there are only four
possible non-ordered pairs $(a_{30,i},b_{30,i})$; thus we will assume that
$i=1,\ldots,4$.
Define $\mu_{30,i}$ and $\nu_{30,i}$ by formula~\eqref{eq:mu-nu}.
\end{itemize}

We also use the following notation.

\begin{itemize}
\item
Let
$$
a_{20}^+=-\frac{19+2\sqrt{95}}{26}+\frac{-30+3\sqrt{95}}{26}\ii,\ b_{20}^+=\frac{-19+2\sqrt{95}}{26}-\frac{30+3\sqrt{95}}{26}\ii.
$$

\item

Let
$$
a_{20}^-=\frac{-19+2\sqrt{95}}{26}+\frac{30+3\sqrt{95}}{26}\ii,\ b_{20}^-=-\frac{19+2\sqrt{95}}{26}+\frac{30-3\sqrt{95}}{26}\ii.
$$

\item
Let $a_{20,1}^+$ and $b_{20,1}^+$ be two different roots of the equation
$$
x^2+(1+\sqrt{6}\ii)x+\sqrt{6}\ii-1=0.
$$

\item
Let $a_{20,2}^+$ and $b_{20,2}^+$ be two different roots of the equation
$$
49x^2+(-63+35\sqrt{6}\ii)x-45\sqrt{6}\ii+39=0.
$$

\item
Let $a_{20,1}^-$ and $b_{20,1}^-$ be two different roots of the equation
$$
x^2+(1-\sqrt{6}\ii)x-\sqrt{6}\ii-1=0.
$$

\item
Let $a_{20,2}^-$ and $b_{20,2}^-$ be two different roots of the equation
$$
49x^2+(-63-35\sqrt{6}\ii)x+45\sqrt{6}\ii+39=0.
$$

\item
Let $a_{30,1}^+$ and $b_{30,1}^+$ be two different roots of the equation
$$
4x^2+(4-4\ii)x-4\ii+1=0.
$$

\item
Let $a_{30,2}^+$ and $b_{30,2}^+$ be two different roots of the equation
$$
676x^2-(52+260\ii)x+20\ii+121=0.
$$

\item
Let $a_{30,1}^-$ and $b_{30,1}^-$ be two different roots of the equation
$$
4x^2+(4+4\ii)x+4\ii+1=0.
$$

\item
Let $a_{30,2}^-$ and $b_{30,2}^-$ be two different roots of the equation
$$
676x^2-(52-260\ii)x-20\ii+121=0.
$$

\item
Let $a_{30}^+$ and $b_{30}^+$ be two different roots of the equation
$$
98x^2+(14+35\sqrt{6}\ii)x+5\sqrt{6}\ii-12=0.
$$

\item
Let $a_{30}^-$ and $b_{30}^-$ be two different roots of the equation
$$
98x^2+(14-35\sqrt{6}\ii)x-5\sqrt{6}\ii-12=0.
$$
\end{itemize}

We denote by $X_{\mu,\nu}$ the double cover of $Q$ branched over
$S_{\mu,\nu}$. Recall that the singularities of $X_{\mu,\nu}$ are
ordinary double points if and only if the same holds for
the singularities of $S_{\mu,\nu}$.
Define the \emph{defect} $\delta(X_{\mu,\nu})$ of the threefold
$X_{\mu,\nu}$ as
$$
\delta(X_{\mu,\nu})=
|\Sing(S_{\mu,\nu})|+\dim H^0\big(\P^4, \mathcal{I}(3)\big)-35,
$$
where $\mathcal{I}\subset\OO_{\P^4}$
is the ideal sheaf of the set $\Sing(S_{\mu,\nu})$.
Corollaries~\ref{corollary:nu},
\ref{corollary:1-1-1-a-b}, and~\ref{corollary:a-a-b-b-1}
provide a complete information about singularities
of the surfaces $S_{\mu,\nu}$
and thus also of the threefolds $X_{\mu,\nu}$.
Note that given a singular point $P$ of
$S_{\mu,\nu}$ one can check whether~$P$ is an ordinary double
point of $S_{\mu,\nu}$ by making an appropriate (analytic) change
of coordinates so that
$Q$ is given in a neighborhood of $P$ by vanishing of some new
coordinate~$z$, and then finding the rank of the
matrix of the second derivatives of the restriction of
the quartic~\eqref{eq:quartic} to the subspace~\mbox{$z=0$}.
Collecting all this information, we obtain the following result.

\begin{corollary}\label{corollary:classification}
Suppose that the surface $S_{\mu,\nu}$ is singular.
Then one of the cases given in Table~\ref{table:singularities}
occurs. In the first column of Table~\ref{table:singularities}
we place the number of singular points~\mbox{$|\Sing(S_{\mu,\nu})|$}.
In the second column we describe the set $\Sing(S_{\mu,\nu})$ itself,
namely, we list the $\SS_5$-orbits whose union gives
$\Sing(S_{\mu,\nu})$.
Note that in the 12th line we refer to non-ordered pairs
$a,b$ such that $a^2+b^2+3=0$, $a\neq 1$, $b\neq 1$, $a\neq b$ and
$a+b+3\neq 0$; similarly, in
the 20th line we refer to non-ordered pairs
$a,b$ such that $2a^2+2b^2+1=0$, $a\neq 1$, $b\neq 1$, $a\neq b$,
and~\mbox{$2a+2b+1\neq 0$}.
The corresponding pairs $(\mu,\nu)$ are
in the third column. In the forth column there are all pairs
$(\mu,\nu)$ for which
the singularities of $S_{\mu,\nu}$ are not just ordinary double points.
Finally, in the last column there is the defect~$\delta(X_{\mu,\nu})$.
\end{corollary}

\setlength{\LTleft}{-20cm plus -1fill}
\setlength{\LTright}{\LTleft}
\begin{longtable}{|c|c|c|c|c|}
\caption{Singularities of $S_{\mu,\nu}$}\label{table:singularities}\\
\hline
$\sharp$  &\vstrut $\Sing(S_{\mu,\nu})$ & $(\mu, \nu)$
& non-ODP $(\mu,\nu)$
& $\delta$ \\
\hline
\hline
\endhead

$5$\vstrut & $\Sigma_5^+$ & $\left(\mu,\left(\frac{8}{25}+\frac{6}{25}\ii\right)\mu+\frac{7}{500}+\frac{6}{125}\ii\right)$ &
\begin{minipage}[c]{3.9cm}
$(\mu_{5,1}^+, \nu_{5,1}^+)$,
$(\mu_{5,2}^+, \nu_{5,2}^+)$,
$\left(\mu_{5}^{ab,+},\nu_{5}^{ab,+}\right)$
\end{minipage}
& $0$\\
\hline
$5$\vstrut& $\Sigma_5^-$ & $\left(\mu,\left(\frac{8}{25}-\frac{6}{25}\ii\right)\mu+\frac{7}{500}-\frac{6}{125}\ii\right)$ &
\begin{minipage}[c]{3.9cm}
$(\mu_{5,1}^-, \nu_{5,1}^-)$,
$(\mu_{5,2}^-, \nu_{5,2}^-)$,
$\left(\mu_{5}^{ab,-},\nu_{5}^{ab,-}\right)$
\end{minipage}
 & $0$\\
\hline
$10$\vstrut & $\Sigma_5^+, \Sigma_5^-$ &
$\left(-\frac{1}{5},-\frac{1}{20}\right)$ &
$\varnothing$
& $0$\\
\hline
$10$\vstrut & $\Sigma_{10}^+$ &
$\left(\mu,\left(\frac{8}{25}+\frac{2\sqrt{6}}{75}\ii\right)\mu+\frac{23}{750}+\frac{2\sqrt{6}}{375}\ii\right)$
&
\begin{minipage}[c]{4.5cm}
$(\mu_{10,1}^+, \nu_{10,1}^+)$,
$(\mu_{10,2}^+, \nu_{10,2}^+)$,
$\left(\mu_{20\to 10}^+,\nu_{20\to 10}^+\right)$,
$\left(\mu_{30\to 10}^+,\nu_{30\to 10}^+\right)$
\end{minipage}
& $0$\\
\hline
$10$\vstrut & $\Sigma_{10}^-$ &
$\left(\mu,\left(\frac{8}{25}-\frac{2\sqrt{6}}{75}\ii\right)\mu+\frac{23}{750}-\frac{2\sqrt{6}}{375}\ii\right)$
&
\begin{minipage}[c]{4.5cm}
$(\mu_{10,1}^-, \nu_{10,1}^-)$,
$(\mu_{10,2}^-, \nu_{10,2}^-)$,
$\left(\mu_{20\to 10}^-,\nu_{20\to 10}^-\right)$,
$\left(\mu_{30\to 10}^-,\nu_{30\to 10}^-\right)$
\end{minipage}
& $0$\\
\hline
$15$\vstrut & $\Sigma_5^+, \Sigma_{10}^+$ & $\left(-\frac{1}{5}-\frac{(9+\sqrt{6})}{120}\ii,
\frac{-16+\sqrt{6}}{500}+\frac{(-9-\sqrt{6})}{375}\ii\right)$ &
$\varnothing$
& $0$\\
\hline
$15$\vstrut & $\Sigma_5^+, \Sigma_{10}^-$ &
$\left(-\frac{1}{5}+\frac{(-9+\sqrt{6})}{120}\ii,
\frac{-16-\sqrt{6}}{500}+\frac{(-9+\sqrt{6})}{375}\ii\right)$
&
$\varnothing$
& $0$\\
\hline
$15$\vstrut & $\Sigma_5^-, \Sigma_{10}^+$ & $\left(-\frac{1}{5}+\frac{(9-\sqrt{6})}{120}\ii,
\frac{-16-\sqrt{6}}{500}+\frac{(9-\sqrt{6})}{375}\ii\right)$ &
$\varnothing$
& $0$\\
\hline
$15$\vstrut & $\Sigma_5^-, \Sigma_{10}^-$ & $\left(-\frac{1}{5}+\frac{(9+\sqrt{6})}{120}\ii,
\frac{-16+\sqrt{6}}{500}+\frac{(9+\sqrt{6})}{375}\ii\right)$ &
$\varnothing$
 & $0$\\
\hline
$20$\vstrut & $\Sigma_{10}^+, \Sigma_{10}^-$ & $\left(-\frac{1}{5},
-\frac{1}{30}\right)$ &
$\varnothing$
& $0$\\
\hline
$20$\vstrut & $\Sigma_{20}$ & $\left(-\frac{1}{3},-\frac{1}{6}\right)$ &
$\varnothing$
& $0$\\
\hline
$20$\vstrut & $\Sigma_{20}^{a,b}$ & $\left(-\frac{(a+b+1)}{3\cdot (a+b+3)},\frac{(a^3+b^3+3)\cdot(a+b+1)-3ab\cdot(a+b+3)}{3\cdot(a+b+3)^4}\right)$ &
$(\mu_{20,i},\nu_{20,i}), 1\leq i\leq 4$
& $0$\\
\hline
$25$\vstrut & $\Sigma_5^+, \Sigma_{20}^{a_{20}^+,b_{20}^+}$ &
$\left(-\frac{1}{5}+\frac{1}{5}\ii,-\frac{49}{500}+\frac{8}{125}\ii\right)$
&
$\varnothing$
 & $0$ \\
\hline
$25$\vstrut & $\Sigma_5^-, \Sigma_{20}^{a_{20}^-,b_{20}^-}$ &
$\left(-\frac{1}{5}-\frac{1}{5}\ii,-\frac{49}{500}-\frac{8}{125}\ii\right)$
&
$\varnothing$
 & $0$\\
\hline
$30$\vstrut &
$\Sigma_{10}^+, \Sigma_{20}^{a_{20,1}^+,b_{20,1}^+}$
&
$\left(-\frac{1}{5}+\frac{\sqrt{6}}{15}\ii,-\frac{11}{250}+\frac{8\sqrt{6}}{375}\ii\right)$
 &
$\varnothing$
 & $0$
 \\
\hline
$30$\vstrut &
$\Sigma_{10}^+, \Sigma_{20}^{a_{20,2}^+,b_{20,2}^+}$
&
$\left(-\frac{1}{5}+\frac{\sqrt{6}}{45}\ii,-\frac{83}{2250}+\frac{8\sqrt{6}}{1125}\ii\right)$
 &
$\varnothing$
 & $0$
 \\
\hline
$30$\vstrut &
$\Sigma_{10}^-, \Sigma_{20}^{a_{20,1}^-,b_{20,1}^-}$
 &
$\left(-\frac{1}{5}-\frac{\sqrt{6}}{15}\ii,-\frac{11}{250}-\frac{8\sqrt{6}}{375}\ii\right)$
 & $\varnothing$
 & $0$
 \\
\hline
$30$\vstrut &
$\Sigma_{10}^-, \Sigma_{20}^{a_{20,2}^-,b_{20,2}^-}$
 &
$\left(-\frac{1}{5}-\frac{\sqrt{6}}{45}\ii,-\frac{83}{2250}-\frac{8\sqrt{6}}{1125}\ii\right)$
 & $\varnothing$
 & $0$
 \\
\hline
$30$\vstrut & $\Sigma_{30}$
&
$\left(-\frac{1}{6},-\frac{1}{48}\right)$
&
$\varnothing$
 & $5$
 \\
\hline
$30$\vstrut & $\Sigma_{30}^{a,b}$ &
$\left(-\frac{(a+b+1)}{3\cdot(2a+2b+1)},\frac{(2a^3+2b^3+1)\cdot(a+b+1)-3ab\cdot(2a+2b+1)}{3\cdot(2a+2b+1)^4}\right)$ &
$(\mu_{30,i},\nu_{30,i}), 1\leq i\leq 4$
& $5$\\
\hline
$35$\vstrut &
$\Sigma_{5}^+, \Sigma_{30}^{a_{30,1}^+,b_{30,1}^+}$
&
$\left(-\frac{2}{15}+\frac{1}{15}\ii,-\frac{67}{1500}+\frac{14}{375}\ii\right)$
 & $\varnothing$
 & $5$ \\
\hline
$35$\vstrut &
$\Sigma_{5}^+, \Sigma_{30}^{a_{30,2}^+,b_{30,2}^+}$
&
$\left(-\frac{4}{15}+\frac{1}{15}\ii,-\frac{131}{1500}+\frac{2}{375}\ii\right)$
 & $\varnothing$
 & $5$ \\
\hline
$35$\vstrut &
$\Sigma_{5}^-, \Sigma_{30}^{a_{30,1}^-,b_{30,1}^-}$
 &
$\left(-\frac{2}{15}-\frac{1}{15}\ii,-\frac{67}{1500}-\frac{14}{375}\ii\right)$
& $\varnothing$
 &  $5$\\
\hline
$35$\vstrut &
$\Sigma_{5}^-, \Sigma_{30}^{a_{30,2}^-,b_{30,2}^-}$
 &
$\left(-\frac{4}{15}-\frac{1}{15}\ii,-\frac{131}{1500}-\frac{2}{375}\ii\right)$
& $\varnothing$
 &  $5$\\
\hline
$40$\vstrut & $\Sigma_{10}^+, \Sigma_{30}^{a_{30}^+,b_{30}^+}$ &
$\left(-\frac{1}{5}-\frac{\sqrt{6}}{30}\ii,-\frac{7}{250}-\frac{4\sqrt{6}}{375}\ii\right)$
 & $\varnothing$
 &  $10$\\
\hline
$40$\vstrut & $\Sigma_{10}^-, \Sigma_{30}^{a_{30}^-,b_{30}^-}$ &
$\left(-\frac{1}{5}+\frac{\sqrt{6}}{30}\ii,-\frac{7}{250}+\frac{4\sqrt{6}}{375}\ii\right)$
 & $\varnothing$
 &  $10$\\
\hline
\end{longtable}

\begin{remark}
Let $C_5^+$, $C_5^-$, $C_{10}$, $C_{10}^-$, $C_{20}^o$, $C_{30}^o$ be curves in
the plane~\mbox{$\mathrm{Spec}\,\C[\mu,\nu]\cong\mathbb{A}^2$}
parameterizing pairs of $(\mu, \nu)$ that correspond to surfaces $S_{\mu,\nu}$ with singularities
at the points of the $\SS_5$-orbits $\Sigma_5^+$, $\Sigma_5^-$, $\Sigma_{10}^+$, $\Sigma_{10}^-$, $\Sigma_{20}^{ab}$,
and $\Sigma_{30}^{ab}$, respectively. Let $C_{20}$, $C_{30}$ be closures of $C_{20}^o$, $C_{30}^o$
in~\mbox{$\mathrm{Spec}\,\C[\mu,\nu]$}, respectively.
Then as suggested by notation one has
\begin{align*}
& C_5^+\cap C_{20}=C_5^+\cap C_{30}=\left(\mu_5^{ab,+},\nu_5^{ab,+}\right),\\
& C_5^-\cap C_{20}=C_5^-\cap C_{30}=\left(\mu_5^{ab,-},\nu_5^{ab,-}\right),\\
& C_{10}^+\cap C_{20}=\left(\mu_{20\to 10}^{+},\nu_{20\to 10}^{+}\right),\  C_{10}^+\cap C_{30}=\left(\mu_{30\to 10}^{+},\nu_{30\to 10}^{+}\right),\\
& C_{10}^-\cap C_{20}=\left(\mu_{20\to 10}^{-},\nu_{20\to 10}^{-}\right),\  C_{10}^-\cap C_{30}=\left(\mu_{30\to 10}^{-},\nu_{30\to 10}^{-}\right).
\end{align*}
Corollary~\ref{corollary:nu}, Remark~\ref{remark:quartic-curve-1}, and Remark~\ref{remark:quartic-curve-2}
show that the locus of pairs $(\mu,\nu)$ corresponding to singular
threefolds $X_{\mu,\nu}$ is a union of four lines and two quartic curves in~\mbox{$\mathrm{Spec}\,\C[\mu,\nu]$}.
\end{remark}

\begin{remark}\label{remark:two-parameter}
The affine plane $\mathbb{A}^2=\mathrm{Spec}\,\C[\mu,\nu]$ is a parameter space
for the threefolds~$X_{\mu,\nu}$, but not a moduli space:
there do exist different pairs $(\mu,\nu)$ and $(\mu',\nu')$ such
that $X_{\mu,\nu}\cong X_{\mu',\nu'}$, and moreover the latter isomorphism is
$\A_5$-equivariant (see Remark~\ref{remark:exceptional-values-for-AM} below
for such examples). However, there is still a two-parameter family
of threefolds $X_{\mu,\nu}$ and a one-parameter family of
singular threefolds $X_{\mu,\nu}$ up to isomorphism. Indeed,
every isomorphism between $X_{\mu,\nu}$ and $X_{\mu',\nu'}$ gives
rise to an element of $\mathrm{Aut}(\P^4)\cong\mathrm{PGL}_5(\C)$
that gives an isomorphism $S_{\mu,\nu}\cong S_{\mu',\nu'}$ provided
that $S_{\mu,\nu}$ and $S_{\mu',\nu'}$ have sufficiently nice singularities
(which is the case for a general singular $S_{\mu,\nu}$ by Corollary~\ref{corollary:classification}).
Choose a surface $S_{\mu,\nu}$ that has at most ordinary double
points as singularities, and suppose that there is a one-parameter family of automorphisms
$$A_t\in\mathrm{PGL}_5(\C)$$
such that $A_t(S_{\mu,\nu})=S_{\mu_t,\nu_t}$
for some $(\mu_t,\nu_t)$. Then the action of $\A_5$ is normalized
by~$A_t$ since the automorphism group of $S_{\mu,\nu}$ is finite,
and thus $A_t$ actually commutes with the action of $\A_5$ since the
automorphism group of $\A_5$ itself is finite.
On the other hand, since $I\oplus W_4$ is a sum of two
irreducible $\A_5$-representations, there is only a one-parameter family
of automorphisms of $\P^4$ that commute with the action of~$\A_5$.
Moreover, we already know such one-parameter family, i.e. the family
of coordinate changes used in the proof of Lemma~\ref{lemma:quadric-lambda}.
But a general automorphism from this family does not preserve the quadric~$Q$,
and thus does not map $S_{\mu,\nu}$ to any of the surfaces $S_{\mu',\nu'}$
since $Q$ is the unique quadric passing through the surface~$S_{\mu',\nu'}$.
\end{remark}

\begin{remark}
Vanishing of the defect for the threefolds $X_{\mu,\nu}$ described
in the first five lines of Table~\ref{table:singularities}
can be obtained not only by a direct computation, but also
from~\cite[Proposition~1.5]{Shramov}.
\end{remark}

Now we are ready to make conclusions on (non-)rationality of the threefolds~$X_{\mu,\nu}$.
It follows from~\cite[Theorem~2]{Cynk} (or rather from the proof of this
theorem) that
$$
\rkCl(X_{\mu,\nu})=1+\delta(X_{\mu,\nu})
$$
provided that $X_{\mu,\nu}$ has only ordinary double points as singularities.
On the other hand, by~\cite[Theorem~1.1]{Shramov}
the variety $X_{\mu,\nu}$ is non-rational provided that
the singularities of~$X_{\mu,\nu}$ are ordinary double
points, and $\rkCl(X_{\mu,\nu})=1$.
Therefore, Corollary~\ref{corollary:classification} implies
the following result.

\begin{corollary}\label{corollary:non-rationality-apart-of-30}
The varieties $X_{\mu,\nu}$ with $|\Sing(X_{\mu,\nu})|\neq 30$ listed in Table~\ref{table:singularities} are non-rational
up to a finite number of (possible) exceptions.
%
%
\end{corollary}

\begin{proof}
Note that there is only a finite number of pairs $(\mu,\nu)$ such that
the singularities of $X_{\mu,\nu}$
are at worse than ordinary double points
(they are listed in the fourth column of Table~\ref{table:singularities}).
Moreover, among $(\mu,\nu)$ such that $|\Sing(X_{\mu,\nu})|\neq 30$ there
is only a finite number of cases when
the defect of $X_{\mu,\nu}$ does not vanish
(see the fifth column of Table~\ref{table:singularities}).
\end{proof}

We will see later in Corollary~\ref{corollary:non-rationality-30}
that the varieties $X_{\mu,\nu}$ with $|\Sing(X_{\mu,\nu})|=30$
are non-rational up to a finite number of possible exceptions as well.

\section{Artin--Mumford-type double covers}
\label{section:Artin-Mumford}

In this section we show that certain double quadrics with an action of the group~$\A_5$
are not stably rational.
We use the notation of~\S\ref{section:proj-perm}.

Note that
\begin{equation}\label{eq:Sym2}
\Sym^2(W_4^\vee)\cong I\oplus W_4\oplus W_5.
\end{equation}
Let $W$ be the (unique) subrepresentation of $\Sym^2(W_4^\vee)$ isomorphic to
$I\oplus W_4$. Then $W$ is identified with the vector subspace in the space
of quadratic forms in variables $x_0,\ldots,x_4$
subject to the relation
$$x_0+\ldots+x_4=0,$$
that is spanned by the forms
$x_i^2$, $0\le i\le 4$. Let $\xi_{ii}$, $0\le i\le 4$, be coordinates in $W$,
so that elements of $W$ are written as $\sum \xi_{ii} x_i^2$.

Put $\P^4=\P(W)$. Then $\xi_{ii}$ are homogeneous coordinates in $\P^4$
that are permuted by the groups $\A_5$ and $\SS_5$.
Put
$$
\sigma_k(\xi_{00},\ldots,\xi_{44})=\xi_{00}^k+\ldots+\xi_{44}^k.
$$
Any (reduced) $\A_5$-invariant quadric in $\P^4$ is given by
\begin{equation}\label{eq:quadric-lambda-xi}
\sigma_2(\xi_{00},\ldots,\xi_{44})+\lambda\sigma_1(\xi_{00},\ldots,\xi_{44})^2=0
\end{equation}
for some $\lambda\in\C$, cf.~\eqref{eq:quadric-lambda}.

Let $Y$ be the quartic in $\P(W)$ given by the vanishing of the determinant
of a quadratic form. Then the equation of $Y$ can be written
as
\begin{multline*}
\det\left(
\begin{array}{cccc}
\xi_{11}+\xi_{00} & \xi_{00} & \xi_{00} &\xi_{00}\\
\xi_{00} & \xi_{22}+\xi_{00} & \xi_{00} &\xi_{00}\\
\xi_{00} & \xi_{00} & \xi_{33}+\xi_{00} &\xi_{00}\\
\xi_{00} & \xi_{00} & \xi_{00} &\xi_{44}+\xi_{00}
\end{array}
\right)=\\=\xi_{11}\xi_{22}\xi_{33}\xi_{44}+\xi_{00}\xi_{22}\xi_{33}\xi_{44}+\xi_{00}\xi_{11}\xi_{33}\xi_{44}+
\xi_{00}\xi_{11}\xi_{22}\xi_{44}+\xi_{00}\xi_{11}\xi_{22}\xi_{33}=0.
\end{multline*}

Denote by $Q_{\lambda}$ the quadric given by
equation~\eqref{eq:quadric-lambda-xi}, and denote by $S_{\lambda}$
the intersection of $Y$ with $Q_{\lambda}$.
Suppose that $Q_{\lambda}$ is smooth.
By Lemma~\ref{lemma:quadric-lambda} this happens if and only if~\mbox{$\lambda\neq -\frac{1}{5}$}.
Let $\alpha$ be a root of the equation
$$
5\alpha^2+2\alpha=\lambda,
$$
and put
\begin{equation}\label{eq:xi-to-xi-prime}
\xi_{ii}^\prime=\xi_{ii}+\alpha\sigma_1(\xi_{00},\ldots,\xi_{44}).
\end{equation}
In particular, one has
$$
\xi_{ii}=\xi_{ii}^\prime-\frac{\alpha}{5\alpha+1}\sigma_1(\xi_{00},\ldots,\xi_{44}),
$$
so the change of variables~\eqref{eq:xi-to-xi-prime} is invertible.
The quadric
$Q_{\lambda}$ is given by the equation
$$
\sigma_2(\xi_{00}^\prime,\ldots,\xi_{44}^\prime)=0
$$
(cf. the proof of Lemma~\ref{lemma:quadric-lambda}),
while the quartic $Y$ is given by the equation
\begin{multline*}
\sigma_4(\xi_{00}^\prime,\ldots,\xi_{44}^\prime)-
4\cdot\frac{3\alpha+1}{3(5\alpha+1)}\cdot\sigma_3(\xi_{00}^\prime,\ldots,\xi_{44}^\prime)\sigma_1(\xi_{00}^\prime,\ldots,\xi_{44}^\prime)-\\
-\frac{165\alpha^4+164\alpha^3+66\alpha^2+12\alpha+1}{6(5\alpha+1)^4}\sigma_1(\xi_{00}^\prime,\ldots,\xi_{44}^\prime)^4+\\
+\frac{11\alpha^2+6\alpha+1}{(5\alpha+1)^2}\sigma_2(\xi_{00}^\prime,\ldots,\xi_{44}^\prime)\sigma_1(\xi_{00}^\prime,\ldots,\xi_{44}^\prime)^2=0.
\end{multline*}
Therefore, the surface $S_{\lambda}$ is isomorphic to
the surface $S_{\mu,\nu}$ in the notation of~\S\ref{section:proj-perm} for
$$
\mu=-\frac{3\alpha+1}{3(5\alpha+1)},\ \nu=-\frac{165\alpha^4+164\alpha^3+66\alpha^2+12\alpha+1}{6(5\alpha+1)^4}.
$$
Applying Corollary~\ref{corollary:classification}, we
obtain the following result.

\begin{lemma}
\label{lemma:exc-lambda}
The singular locus of the surface $S_{\lambda}$ is a single
$\SS_5$-orbit
of twenty ordinary
double points, provided that
\begin{equation}\label{eq:exc-lambda}
\lambda\in\C\setminus\{-\frac{1}{5}, 
-1, -\frac{1}{2}\}.
\end{equation}
\end{lemma}

\begin{remark}\label{remark:exceptional-values-for-AM}
Using Table~\ref{table:singularities}, one can describe
singularities of the surfaces $S_{\lambda}$ for exceptional values
$\lambda\in\{-1, -\frac{1}{2}\}$.
Taking $\alpha=-\frac{1}{5}-\frac{2}{5}\ii$
and $\alpha=-\frac{1}{5}+\frac{2}{5}\ii$
gives isomorphisms
$$S_{\mu_5^{ab,+},\nu_5^{ab,+}}
\cong S_{-1}\cong
S_{\mu_5^{ab,-},\nu_5^{ab,-}}$$
respectively, so that $S_{-1}$ has exactly five
singular points, and all of them are worse than ordinary double ones. Taking
$\alpha=-\frac{1}{5}-\frac{\sqrt{6}}{10}\ii$
and
$\alpha=-\frac{1}{5}+\frac{\sqrt{6}}{10}\ii$
gives isomorphisms
$$
S_{\mu_{20\to 10}^+,\nu_{20\to 10}^+}\cong
S_{-\frac{1}{2}}\cong
S_{\mu_{20\to 10}^-,\nu_{20\to 10}^-}
$$
respectively, so that $S_{-\frac{1}{2}}$ has
exactly ten singular points, and all of them are worse than ordinary double ones.
\end{remark}

Let $\Delta$ be the determinantal hypersurface
in
$$\P^9=\P\big(\Sym^2(W_4^\vee)\big),$$
i.e. the hypersurface parameterizing singular
quadrics in $\P^3=\P(W_4)$.
Let $\Delta_i\subset\Delta$, $i=1,2$, be the subvariety
parameterizing quadrics of rank at most~$3-i$.
Then $\Delta$ is singular
along $\Delta_1$, the singularity of
$\Delta$ at every point $P\in\Delta_1\setminus\Delta_2$
is locally isomorphic to a product of a germ of a surface ordinary double
point with $\mathbb A^6$, and
the singularity of $\Delta$
at every point of $\Delta_2$ has multiplicity~$3$. Applying Lemma~\ref{lemma:exc-lambda},
we obtain the following result.

\begin{corollary}\label{corollary:sing-S-lambda}
Suppose that $\lambda$ is like in~\eqref{eq:exc-lambda}.
Then the surface $S_{\lambda}$ intersects the subvariety
$\Delta_1\subset\Delta$ transversally at $20$ points,
and has no singular points outside~$\Delta_1$.
In particular, $Q_{\lambda}$ intersects
$\Delta$ transversally at the points of $\Delta\setminus\Delta_1$, intersects
$\Delta_1$ transversally at the points of $\Delta_1\setminus\Delta_2$,
and is disjoint from~$\Delta_2$.
\end{corollary}

Let $T$ be a subvariety of $\P^9\cong\P\big(\Sym^2(W_4^\vee)\big)$.
Let $\phi\colon\mathcal{T}\to T$ be the
restriction of the tautological quadric bundle over
$\P^9$ to $T$.
Our current goal is to find conditions on $T$ to
guarantee that $\mathcal{T}$ is smooth. This is due to the fact that
in our construction we restrict ourselves to a certain subfamily of
quadric threefolds in~$\P^9$; so that we have to check smoothness
of (total spaces of) corresponding bundles explicitly as opposed to more standard approaches that make use
of generality assumptions, see e.g.~\cite[Exercise~7.3.2(i)]{GS}.

\begin{lemma}\label{lemma:det-smooth}
Let $P$ be a point of $T$, and $\mathcal{T}_P$ be the fiber of
$\phi$ over $P$. The following assertions hold.
\begin{itemize}
\item[(i)]
Suppose that $P\not\in\Delta$. Then $\mathcal{T}$ is smooth at every
point of $\mathcal{T}_P$ if and only if $T$ is smooth at $P$.

\item[(ii)]
Suppose that $P\in\Delta\setminus\Delta_1$. Then $\mathcal{T}$ is smooth at every
point of $\mathcal{T}_P$ if and only if $T$ is smooth at $P$, and $T$
intersects $\Delta$ transversally at~$P$.

\item[(iii)]
Suppose that $P\in\Delta_1\setminus\Delta_2$. Then $\mathcal{T}$ is smooth at every
point of $\mathcal{T}_P$ provided that $T$ is smooth at $P$, and $T$
intersects $\Delta_1$ transversally at~$P$.
\end{itemize}
\end{lemma}
\begin{proof}
Choose homogeneous coordinates $z_0,\ldots,z_3$ in $\P^3$
and homogeneous coordinates $\zeta_{ij}$, $0\le i\le j\le 3$, in~$\P^9$.
Suppose that $T$ is given in $\P^9$ by equations
\begin{equation}\label{eq:T}
F_1(\zeta_{ij})=\ldots=F_k(\zeta_{ij})=0.
\end{equation}
Then $\mathcal{T}$ is given in $\P^9\times\P^3$ by equations~\eqref{eq:T} and the equation
\begin{equation}\label{eq:T-upstairs}
\sum\limits_{0\le i\le j\le 3} \zeta_{ij}z_iz_j=0.
\end{equation}

Let $r$ be the codimension of $T$ in $\P^9$, so that $\mathcal{T}$ has codimension~\mbox{$r+1$}
in~\mbox{$\P^9\times\P^3$}.
The variety $\mathcal{T}$ is non-singular at a point $\PP\in\mathcal{T}_P$
if and only if the matrix $M$ of partial derivatives of equations~\eqref{eq:T} and~\eqref{eq:T-upstairs}
with respect to the variables $\zeta_{ij}$ and $z_i$
has rank~\mbox{$r+1$} at~$\PP$. If $T$ is singular at $P$, then the matrix of partial derivatives
of equations~\eqref{eq:T} has rank at most $r-1$ at $P$, so that $\mathcal{T}$ is singular at every
point of the fiber~$\mathcal{T}_P$.

From now on we assume that $T$ is non-singular at $P$.
Note that the partial derivatives of~\eqref{eq:T} with respect to
variables $z_i$ vanish at every point,
while for a smooth point of $\mathcal{T}_P$ there exists a partial derivative
of~\eqref{eq:T-upstairs} with respect to some variable $z_i$ that does not
vanish at that point. Thus we see that the matrix $M(\PP)$
has rank $r+1$ provided that the quadric $\mathcal{T}_P$ is non-singular at~$\PP$.
In particular, this proves assertion~(i).

Suppose that $P\in\Delta\setminus\Delta_1$. We have already seen
that $\mathcal{T}$ is smooth at $\PP$ provided that~$\PP$ is a smooth
point of~$\mathcal{T}_P$. So we suppose that $\PP$ is the (unique)
singular point of~$\mathcal{T}_P$. After a suitable change of
coordinates $z_0,\ldots,z_3$ we may also assume that the point~$P$
corresponds to the quadratic form
\begin{equation*}
z_0^2+z_1^2+z_2^2=0,
\end{equation*}
and $\PP$ corresponds to the point $(0:0:0:1)$ in~$\P^3$.
The only partial derivative of~\eqref{eq:T-upstairs} that does not vanish at~$\PP$
is the one with respect to~$\zeta_{33}$. Similarly, the only partial derivative
of the equation of $\Delta$ that does not vanish at the point $P$
is the one with respect to~$\zeta_{33}$. Therefore, singularity of $\mathcal{T}$ at $\PP$
and transversality of $T$ and $\Delta$ at $P$
are given by the same condition. This proves assertion~(ii).

Now suppose that $P\in\Delta_1\setminus\Delta_2$, and $\PP$ is a
singular point of~$\mathcal{T}_P$. After a suitable change of
coordinates $z_0,\ldots,z_3$ we may also assume that the point $P$
corresponds to the quadratic form
\begin{equation*}
z_0^2+z_1^2=0,
\end{equation*}
and $\PP$ corresponds to the point $(0:0:\alpha:\beta)$ in~$\P^3$.
The partial derivatives of~\eqref{eq:T-upstairs} at $\PP$ with respect to variables
$\zeta_{22}$, $\zeta_{23}$, and $\zeta_{33}$ equal $\alpha^2$, $\beta^2$ and $\alpha\beta$,
respectively, while all other partial derivatives of~\eqref{eq:T-upstairs}
vanish at $\PP$. The subvariety $\Delta_1\subset\P^9$ is given by vanishing of the
\mbox{$3\times 3$}-minors of a matrix of a quadratic form. All partial derivatives of these
minors except those with respect to $\zeta_{22}$, $\zeta_{23}$, and $\zeta_{33}$ vanish at
the point~$P$.
Therefore, transversality of $T$ and $\Delta_1$ at $P$ implies that $\mathcal{T}$
is non-singular at~$\PP$. This proves assertion~(iii).
\end{proof}

\begin{remark}
It would be interesting to have an ``if and only if'' condition
for smoothness of the variety $\mathcal{T}$ in terms
of the varieties $T$, $\Delta$ and~$\Delta_i$.
\end{remark}

In the remaining part of this section we
suppose that $\lambda$ is like in~\eqref{eq:exc-lambda}.
Let $X_{\lambda}$ be a double cover of $Q_{\lambda}$ branched over
$S_{\lambda}$. Put
$$X_{\lambda}^o=X_{\lambda}\setminus\Sing(X_{\lambda}),\quad Q_{\lambda}^o=Q_{\lambda}\setminus\Sing(S_{\lambda}) .$$
Then $X_{\lambda}^o$ is a double cover of $Q_{\lambda}^o$. Let
$\psi\colon\mathcal{Q}_{\lambda}\to Q_{\lambda}$ be the
restriction of the tautological quadric bundle over
$\P\big(\Sym^2(W_4^\vee)\big)$ to $Q_{\lambda}$.
Corollary~\ref{corollary:sing-S-lambda} shows that
$X_{\lambda}^o$ can be identified with a variety parameterizing
families of lines on quadrics corresponding to the points
of~$Q_{\lambda}^o$, and that there
is a natural $\P^1$-bundle $\pi\colon\mathcal{P}_{\lambda}^o\to X_{\lambda}^o$
such that the points of $\mathcal{P}_{\lambda}^o$
parameterize the lines on such quadrics.
The following result is
obtained in a way identical to a well-known approach
to rationality of double solids
(cf.~\cite[Appendix]{Aspinwall},
\mbox{\cite[Lemma~3.2]{IlievKatzarkovPrzyalkowski}},
\cite[Exercise~7.3.2]{GS}).

\begin{lemma}[{cf.~\cite[Lemma~3.2]{IlievKatzarkovPrzyalkowski}}]
\label{lemma:no-section}
The $\P^1$-bundle $\pi$ is not a projectivization of
a vector bundle.
\end{lemma}
\begin{proof}
Suppose that $\pi$ is a projectivization of some vector bundle.
Then there exists a rational section $\sigma\colon X_{\lambda}^o\dasharrow\mathcal{P}_{\lambda}^o$
of $\pi$. Note that $\sigma$ defines a family $\Sigma$ of lines on
quadrics corresponding to the points
of $Q_{\lambda}^o$, and for a (smooth) quadric $M$ corresponding
to a general point of $Q_{\lambda}^o$ there are exactly two lines in $\Sigma$
that are contained in $M$, one of them in each of the two one-parameter
families of lines on~$M$. Taking an intersection point of the latter
two lines, we define a rational section $\sigma'$ of the quadric bundle~$\psi$.
Let $R\subset\mathcal{Q}_{\lambda}$ be the closure of the image of $\sigma'$.
Then for a fiber $F$ of $\psi$ one has an intersection $R\cdot F=1$
on~$\mathcal{Q}_{\lambda}$.

The fivefold $\mathcal{Q}_{\lambda}$ is naturally embedded
into the sixfold $Q_{\lambda}\times\P^3$ as a divisor of
bi-degree~\mbox{$(1,2)$}.
Moreover, $\mathcal{Q}_{\lambda}$ is smooth
by Lemma~\ref{lemma:det-smooth} and Corollary~\ref{corollary:sing-S-lambda}.

By the Lefschetz hyperplane section theorem, there is an element
$$\bar{R}\in H^4(Q_{\lambda}\times\P^3,\mathbb{Z})$$
that restricts to the element $R\in H^4(\mathcal{Q}_{\lambda},\mathbb{Z})$.
The fiber $F$ of $\psi$ can be considered both as an element
of $H^6(\mathcal{Q}_{\lambda},\mathbb{Z})$ and of
$H^8(Q_{\lambda}\times\P^3,\mathbb{Z})$. In the latter group
it is divisible by~$2$.
The intersection of $F$ with $R$ on $\mathcal{Q}_{\lambda}$
equals the intersection of $F$ with $\bar{R}$ on $Q_{\lambda}\times\P^3$.
On the other hand, the former
intersection equals $1$ while the latter intersection is even, which
gives a contradiction.
\end{proof}

\begin{corollary}[{cf.~\cite[Theorem~3.3]{IlievKatzarkovPrzyalkowski}}]
\label{corollary:non-stably-rational}
Let $\tilde{X}_{\lambda}$ be a blow up of
the singular points of~$X_{\lambda}$. Then
there is a non-trivial torsion element in the cohomology group~\mbox{$H^3(\tilde{X}_{\lambda},\mathbb{Z})$}.
In particular,
the varieties $\tilde{X}_{\lambda}$ and
$X_{\lambda}$ are not stably rational.
\end{corollary}
\begin{proof}
By Lemma~\ref{lemma:no-section} the $\P^1$-bundle $\pi$ defines a non-trivial
element in the Brauer group
$$\mathrm{Br}(X_{\lambda}^o)\cong\mathrm{Br}(\tilde{X}_{\lambda}).$$
The latter gives a non-trivial torsion element in $H^3(\tilde{X}_{\lambda},\mathbb{Z})$,
see e.g.~\cite[Appendix]{Aspinwall}.
By~\cite[Proposition~1]{ArtinMumford} such element
provides an obstruction to stable rationality of $\tilde{X}_{\lambda}$, and
thus also of~$X_{\lambda}$.
\end{proof}

Applying Corollary~\ref{corollary:non-stably-rational}
together with~\cite[Theorem~0.1]{Voisin}, we
prove Theorem~\ref{theorem:main}(iii). In fact, this gives the
the following more general result
(that is actually well-known to experts).

\begin{proposition}
Let $X$ be a very general double cover of a (smooth) three-dimensional
quadric branched over an intersection with a quartic. Then
$X$ is not stably rational.
\end{proposition}

\begin{remark}\label{remark:vse-ukradeno-do-nas}
Since double covers of quadrics branched over intersections with
quartics are degenerations of quartic hypersurfaces,
Corollary~\ref{corollary:non-stably-rational} implies
that a very general quartic threefold is stably non-rational.
This result is known from~\cite{CTP}, see also~\cite{HT},
where another approach was used to obtain it.
Moreover, the approach to stable non-rationality of
quartic threefolds via double quadrics was used
in~\cite{ST}, and is actually not affected by a gap in~\cite{ST}
that becomes crucial only in higher dimensions.
\end{remark}

\section{Projectivization of the irreducible representation}
\label{section:proj-irreducible}

In this section we study the unique double quadric with an action of the
group~$\A_5$ such that the corresponding quadric
is embedded into the projectivization of the irreducible
five-dimensional representation of the group~$\A_5$,
i.e. the variety~$X_{irr}$ in the notation of Theorem~\ref{theorem:main}.
This will also appear to be the unique double quadric with an action of the
group~$\A_6$.

Recall that $W_5$ denotes the unique five-dimensional
irreducible representation of the group $\A_5$.
Note that $W_5$ can be also considered as a representation
of the groups~$\A_6$ and~$\SS_6$.
Put $\P^4=\P(W_5)$.
One has
\begin{equation}\label{eq:quadrics-in-W5}
\dim\Hom\big(I,\Sym^2(W_5^\vee)\big)=1, \quad
\dim\Hom\Big(I,\Sym^4(W_5^\vee)\big)=2,
\end{equation}
where $I$ and $W_5$ are considered as representations
of either of the groups $\A_5$, $\A_6$, or~$\SS_6$.
In particular, in $\P^4$ there is a unique quadric $Q$ invariant
under the group~$\A_5$ (or $\A_6$, or~$\SS_6$, respectively),
and a unique intersection $S$ of $Q$ with a quartic invariant
under the group~$\A_5$ (or $\A_6$, or $\SS_6$, respectively).
They can be described as follows.

Let $W\cong I\oplus W_5$ be the six-dimensional permutation representation
of the group~$\SS_6$.
Put $\P^5=\P(W)$, and let $y_0,\ldots,y_5$ be the homogeneous
coordinates in $\P^5$ that are permuted by $\SS_6$. Put
\begin{equation}\label{eq:sigma-in-W5}
\sigma_k(y_0,\ldots,y_5)=y_0^k+\ldots+y_5^k.
\end{equation}
Equation $\sigma_1=0$ defines the linear subspace $\P^4\subset\P^5$.
The quadric $Q$ is defined by an equation
$$
\sigma_1=\sigma_2=0,
$$
and reduced $\SS_6$-invariant (respectively, $\A_6$-invariant,
$\A_5$-invariant) quartics are defined by equations
\begin{equation}\label{eq:W5-quartics}
\sigma_1=\sigma_4-\zeta\sigma_2^2=0, \quad \zeta\in\C.
\end{equation}
In particular, the surface $S$ is defined
by equations
$$
\sigma_1=\sigma_2=\sigma_4=0.
$$

\begin{remark}
The quartics given by equation~\eqref{eq:W5-quartics}
were studied in~\cite{Beauville-S6} and~\cite{CheltsovShramov2015}.
\end{remark}

Consider the point
$$P=(1:1:\omega:\omega:\omega^2:\omega^2)\in\P^5,$$
where $\omega$ is a non-trivial cubic root of $1$.
Let $\Xi$ be the $\SS_6$-orbit of the point $P$.
Then one has~\mbox{$|\Xi|=30$}. Choose a subgroup
$\A_5^\prime\subset\SS_6$ that is isomorphic to $\A_5$ but is not conjugate
to the initial $\A_5\subset\SS_6$.
The restriction of the representation $W_5$ of
$\SS_6$ to $\A_5^\prime$ is isomorphic to $I\oplus W_4$, and
one can assume that the coordinates $y_1,\ldots,y_5$ are permuted
by $\A_5^\prime$. In these coordinates
$\Xi$ is the $\A_5^\prime$-orbit of
the point
$$P^\prime=(1:\omega:\omega:\omega^2:\omega^2)\in\P^4,$$
and the quadric $Q$ is given by equation~\eqref{eq:quadric-lambda} with $\lambda=1$.
Therefore, in the notation of~\S\ref{section:proj-perm}
the $\A_5^\prime$-orbit $\Xi$
is identified with $\Sigma_{30}^{a,b}$
for
$$
a=\frac{-2+\sqrt{6}}{4}+\frac{2\sqrt{3}-\sqrt{2}}{4}\ii,\ b= \frac{-2+\sqrt{6}}{4}+\frac{-2\sqrt{3}+\sqrt{2}}{4}\ii,
$$
and the surface $S$ is identified with
$S_{\mu,\nu}$ for
\begin{equation}\label{eq:mu-nu-for-A6}
\mu=-\frac{6+\sqrt{6}}{30},\
\nu=-\frac{3+8\sqrt{6}}{750}.
\end{equation}
In particular, by Corollary~\ref{corollary:classification}
one has $\Sing(S)=\Xi$, and
the singularities of the surface $S$ are ordinary double points.

Let $X$ be a double cover of $Q$ branched over $S$.

\begin{remark}\label{remark:A5-vs-A5}
By Corollary~\ref{corollary:classification} the singularities of $X$ are $30$ ordinary double points.
The threefold $X$
is the unique double quadric with an
action of the group $\A_6$, see Lemma~\ref{lemma:Prokhorov}.
In the notation of~\S\ref{section:proj-perm}
there is an isomorphism $X\cong X_{\mu,\nu}$ for $\mu$ and $\nu$
given by~\eqref{eq:mu-nu-for-A6}, but this isomorphism is
not $\A_5$-equivariant.
\end{remark}

Below we will need the following notation. Consider the weighted projective space
$\P=\P(1^5,2)$ with weighted homogeneous coordinates $y_1,\ldots,y_5,u$, where
$y_i$ have weight~$1$, and $u$ has weight~$2$. The group $\SS_6$ acts on $\P$
so that $H^0(\mathcal{O}_{\P}(1))$ is identified with the $\SS_6$-representation~$W_5$.
Put
$$
y_0=-(y_1+\ldots+y_5),
$$
and define the forms $\sigma_i(y_1,\ldots,y_5)=\sigma_i(y_0,y_1,\ldots,y_5)$
by formula~\eqref{eq:sigma-in-W5}.

\begin{proposition}[{cf.~\cite{Beauville-S6}}]
\label{proposition:Jac-non-trivial}
The intermediate Jacobian of $X$ is not isomorphic to a product of Jacobians
of curves as a principally polarized abelian variety.
In particular, $X$ is non-rational.
\end{proposition}
\begin{proof}
One can write equations of the threefold $X$ in $\P$ as
\begin{equation}\label{eq:theta-0}
\sigma_2=u^2-\sigma_4=0.
\end{equation}
Consider the equations
\begin{equation}\label{eq:theta}
\sigma_2-\theta u=u^2-\sigma_4=0, \quad \theta\in\C.
\end{equation}
If $\theta\neq 0$, we can rewrite~\eqref{eq:theta}
as
\begin{equation}\label{eq:theta-ne-0}
u-\theta^{-1}\sigma_2=\sigma_4-\theta^{-2}\sigma_2^2=0.
\end{equation}
The latter equations define a threefold $X_{\theta}$ that is isomorphic
to a quartic given by~\eqref{eq:W5-quartics} for $\zeta=\theta^{-2}$.
On the other hand, if $\theta=0$, then~\eqref{eq:theta}
is rewritten as~\eqref{eq:theta-0}. This shows that the
quartic threefolds given by~\eqref{eq:W5-quartics} or~\eqref{eq:theta-ne-0}
degenerate to the double quadric~$X$.
All varieties $X_{\theta}$, $\theta\neq 0$, and $X$ are $\SS_6$-invariant.
Note that the singularities of a general threefold $X_{\theta}$, as well
as the singularities of $X$, are $30$ ordinary double points
that form the $\SS_6$-orbit $\Xi''$ of the point $P''\in\P$ with coordinates
$$
y_1=1,\ y_2=y_3=\omega, \ y_4=y_5=\omega^2,\ u=0,
$$
see~\cite{Beauville-S6} for details.

Put $\Pi=\P\times\mathbb{A}^1$, and let $\mathcal{X}\subset\Pi$ be
the subvariety
$$
\mathcal{X}=\{(P,\theta)\mid \sigma_2(P)-\theta u=u^2-\sigma_4(P)=0\}.
$$
Let $\pi\colon\Pi\to\mathbb{A}^1$ be the natural projection, and put
$\mathcal{X}_{\theta}=\pi^{-1}(\theta)$. Then $\mathcal{X}_{\theta}\cong X_{\theta}$
for $\theta\neq 0$, and $\mathcal{X}_0\cong X$.

Let $\varphi\colon\tilde{\Pi}\to\Pi$
be the blow up of the locus $\Xi''\times\mathbb{A}^1\subset\Pi$, and
let $\tilde{\mathcal{X}}$ be the proper transform of $\mathcal{X}$
on $\tilde{\Pi}$. Denote by $\tilde{\mathcal{X}}_{\theta}$ the proper transform of
$\mathcal{X}_{\theta}$ on $\tilde{\Pi}$. Then $\tilde{\mathcal{X}}_{\theta}$
is smooth for a general $\theta\in\mathbb{A}^1$, and
$\tilde{\mathcal{X}}_0$ is smooth as well. Put
$$
C=\mathbb{A}^1\setminus\{\theta\mid \tilde{\mathcal{X}}_{\theta} \text{\ is singular}\},
$$
and denote by $\tilde{\mathcal{X}}^o$ the preimage of $C$ in $\tilde{\mathcal{X}}$.

The fibration $\tilde{\mathcal{X}}^o\to C$ defines a vector bundle
$\mathcal{W}\to C$ whose fiber $\mathcal{W}_{\theta}$ over $\theta\in C$
is identified with the vector space $H^{2,1}(\tilde{\mathcal{X}}_{\theta})$,
see e.g.~\cite[\S10.2.1]{Voisin-book-I}.
Moreover, there is a fiberwise action of the group $\SS_6$
on $\tilde{\mathcal{X}}^o$. It gives rise to a fiberwise
action of $\SS_6$ on $\mathcal{W}$. By~\cite[Proposition]{Beauville-S6}
one has $\mathcal{W}_{\theta}\cong W_5$ for $\theta\neq 0$.
Note that representations of finite groups do not vary in families,
since there is only a finite number of (characters of) representations of
fixed dimension of a given group. Therefore,
we have $\mathcal{W}_0\cong W_5$. Now the assertion of the proposition follows
by an argument of~\cite[\S3]{Beauville}.
\end{proof}

\begin{remark}
An alternative proof of Proposition~\ref{proposition:Jac-non-trivial}
could be given following the method of~\cite{Beauville-S6} step by step, since
the singular loci of the variety $X$ and the quartics considered in~\cite[\S2]{Beauville},
i.e. the quartics given by~\eqref{eq:W5-quartics},
arise from the same $\SS_6$-orbit in~$\P^4$.
\end{remark}

\begin{remark}
Recall from~\eqref{eq:Sym2}
that there is a unique $\A_5$-subrepresentation isomorphic to $W_5$
in $\Sym^2(W_4^\vee)$.
Thus we can conclude from~\eqref{eq:quadrics-in-W5} that $X$
is a double cover of the $\A_5$-invariant quadric $Q$ branched over an
intersection with the determinantal quartic~\mbox{$\Delta\subset\P\big(\Sym^2(W_4^\vee)\big)$}.
However, the approach of~\S\ref{section:Artin-Mumford}
is not applicable to prove non-rationality of~$X$ since its
branch surface~\mbox{$S\subset Q$} has singularities outside the singular locus~\mbox{$\Delta_1\subset\Delta$}, cf. Lemma~\ref{lemma:det-smooth}(ii).
\end{remark}

Proposition~\ref{proposition:Jac-non-trivial}
proves Theorem~\ref{theorem:main}(iv).
Also, Lemma~\ref{lemma:Prokhorov}, Remark~\ref{remark:A5-vs-A5},
and Proposition~\ref{proposition:Jac-non-trivial}
imply Proposition~\ref{proposition:main}.
Another consequence of Proposition~\ref{proposition:Jac-non-trivial}
is the following result.

\begin{corollary}\label{corollary:non-rationality-30}
The varieties $X_{\mu,\nu}$ with $|\Sing(X_{\mu,\nu})|=30$
listed in Table~\ref{table:singularities}
are non-rational up to a finite number of (possible) exceptions.
\end{corollary}
\begin{proof}
Every double quadric $X_{\mu,\nu}$ is naturally embedded
into the weighted projective space $\P$, so that $X_{\mu,\nu}\subset\P$
is given by equations
$$
\sigma_2=u^2-\sigma_4(y_1,\ldots,y_4)+4\mu \sigma_3(y_1,\ldots,y_4)\sigma_1(y_1,\ldots,y_4)
+\nu \sigma_1(y_1,\ldots,y_4)^4=0.$$
Define a (locally closed) curve $B\subset\mathrm{Spec}\,\C[\mu,\nu]\cong\mathbb{A}^2$ as
\begin{multline*}
B=\left\{(\mu,\nu)\mid \text{\ there are\ } a,b\in\C\text{\ such that\ }
\Sing(S_{\mu,\nu})=\Sigma_{30}^{a,b}, \right.\\
\left. \text{\ and the singularities of\ } S_{\mu,\nu}
\text{\ are ordinary double points} \right\}.
\end{multline*}
By Remark~\ref{remark:A5-vs-A5} the curve $B$ contains the point
$(\mu,\nu)$ given by~\eqref{eq:mu-nu-for-A6}.

Put $\Upsilon=\P\times B$. Define $\mathcal{Z}\subset\Upsilon$
as
$$
\mathcal{Z}=\{(P,\mu,\nu)\mid P\in X_{\mu,\nu}\}.
$$

Let $\phi\colon\tilde{\Upsilon}\to\Upsilon$
be the blow up of the locus
$$
\Sigma=\{(P,\mu,\nu)\mid P\in\Sing(X_{\mu,\nu})\}\subset\Upsilon,
$$
and let $\tilde{\mathcal{Z}}$ be the proper transform of $\mathcal{Z}$
on $\tilde{\Upsilon}$. Denote by $\tilde{\mathcal{Z}}_{\mu,\nu}$ the fiber of
the natural projection $\tilde{\mathcal{Z}}\to B$ over a point $(\mu,\nu)\in B$.
Then $\tilde{\mathcal{Z}}_{\mu,\nu}$
is smooth.

We see that $\tilde{\mathcal{Z}}$ is a family of resolutions of singularities
of the threefolds $X_{\mu,\nu}$ degenerating
to a resolution of singularities of the threefold~$X$.
On the other hand, we know from Proposition~\ref{proposition:Jac-non-trivial}
that the intermediate Jacobian of $X$ is not isomorphic to a product of Jacobians
of curves as a principally polarized abelian variety.
Now the required assertion follows from~\cite[Lemme~5.6.1]{Beauville} applied to the
family~$\tilde{\mathcal{Z}}$.
\end{proof}

Corollaries~\ref{corollary:non-rationality-apart-of-30}
and~\ref{corollary:non-rationality-30}
prove Theorem~\ref{theorem:main}(ii).

\end{document}